\documentclass[11pt]{amsart}

\usepackage{amsmath,amstext,amssymb,amsfonts,amsopn,amsthm,mathrsfs,wasysym,dsfont,comment,bigints,tensor,upgreek} 
\usepackage{longtable,tabularx,booktabs}
\usepackage{enumitem}
\usepackage{color}
\usepackage{latexsym}
\usepackage{mathtools}
\usepackage{float}
\usepackage{textcomp}

\mathtoolsset{showonlyrefs}

\usepackage[dvipsnames]{xcolor}
\usepackage{tikz}
\usetikzlibrary{patterns,arrows}
\usepackage[
colorlinks=true,
linkcolor=magenta,
citecolor=magenta,
urlcolor=magenta,
pagebackref,
]{hyperref}

\numberwithin{equation}{section}

\usepackage{version}
\excludeversion{comentario}
\excludeversion{demostraciones-fuera}
\includeversion{demostraciones-dentro}

\usepackage[bb=boondox]{mathalfa}

\def\0{\mathbb{0}}

\def\R{\mathbb{R}}

\def\Z{\mathbb{Z}}

\def\S{\mathbb{S}}

\def\F{\mathcal{F}}
\DeclareMathOperator*{\supp}{supp}

\newtheorem{theorem}{Theorem}[section]
\newtheorem{proposition}[theorem]{Proposition}
\newtheorem{lemma}[theorem]{Lemma}

\theoremstyle{definition}
\newtheorem{definition}[theorem]{Definition}
\newtheorem{remark}[theorem]{Remark}

\newcommand{\im}{\operatorname{Im}}
\newcommand{\re}{\operatorname{Re}}

\definecolor{myyellow}{rgb}{0.9290 0.6940 0.1250}
\definecolor{myorange}{rgb}{0.8500 0.3250 0.0980}
\definecolor{myred}{rgb}{0.6350 0.0780 0.1840}
\definecolor{mygreen}{rgb}{0.4660 0.6740 0.1880}
\definecolor{mycyan}{rgb}{0.3010 0.7450 0.9330}
\definecolor{myblue}{rgb}{0 0.4470 0.7410}
\definecolor{myred}{rgb}{0.4940 0.1840 0.5560}

\allowdisplaybreaks

\author[J.\ Canto]{Javier Canto}
\address[Javier Canto]{
Department of Mathematics, University of the Basque Country UPV/EHU, 48490 Leioa, Spain
}
\email{\href{mailto:javier.canto@ehu.eus}{javier.canto@ehu.eus}}

\author[N. M. Schiavone]{Nico Michele Schiavone}
\address[Nico M. Schiavone]{
	Departamento de Matemática e Informática Aplicadas a la Ingeniería Civil y Naval, ETSI Caminos, Canales y Puertos, Universidad Politécnica de Madrid,
	28040 Madrid, Spain}
\email{\href{mailto:nico.schiavone@upm.es}{nico.schiavone@upm.es}}

\author[L.\ Vega]{Luis Vega}
\address[Luis Vega]{
Department of Mathematics, University of the Basque Country UPV/EHU, 48490 Leioa, Spain  \&
BCAM -- Basque Center for Applied Mathematics, 48009 Bilbao, Spain}
\email{\href{mailto:luis.vega@ehu.eus}{luis.vega@ehu.eus}; \href{mailto:lvega@bcamath.org}{lvega@bcamath.org}}

\title[Asymptotic fractional uncertainty principle for the Helmholtz equation]{Asymptotic fractional uncertainty principle for the Helmholtz equation with periodic scattering data}


\makeatletter
\@namedef{subjclassname@2020}{%
  \textup{2020} Mathematics Subject Classification}
\makeatother


\date{\today}

\begin{document}

\begin{abstract}
	We investigate the fractional dispersion of solutions to the Helmholtz equation with periodic scattering data. We show that, under appropriate rescaling, the interaction between the different frequencies exhibits the same fluctuating behavior as for the Schr\"odinger equation. To achieve this, we first establish an asymptotic fractional uncertainty principle for solutions to the Helmholtz equation.
\end{abstract}

\maketitle


\section{Introduction and main results}
\label{sec:intro}

In this article, we study the dispersion behavior of high-frequency plane waves propagating in a given direction. Our goal is to determine to what extent the results for the free Schr\"odinger equation remain valid in the high-frequency limit. Specifically, we aim to replicate the results obtained by Kumar, Ponce Vanegas, and Vega \cite{MR4469234} for the free Schr\"odinger equation in the case of the Helmholtz equation.
In \cite{MR4469234}, the authors consider low-regularity solutions to the free Schr\"odinger equation on the line
\begin{equation}
    \label{eq:Schr-initial-problem}
    \begin{cases}
        \partial_t u(x,t) =i\partial_{xx}u(x,t), & x\in \R, \: t>0,\\
        u(x,0)=u_0(x), & x\in \R,
    \end{cases}
\end{equation}
where the initial data $u_0\in L^2(\R)$ satisfies the regularity conditions
\[
\int_{\R}|x|^{2b}|u_0(x)|^2dx <\infty, \qquad \int_{\R} |\xi|^{2b}|\hat u_0(\xi)|^2d\xi<\infty,
\]
for some $b\in (0,1)$. For these solutions, they study the dispersion of order $b$ given by the expression
\begin{equation}
    \label{eq:hb-def-schro}
    \mathsf{h}_b[u_0](t) = \int_{\R} |x|^{2b}|u(x,t)|^2dx,
\end{equation}
where $u$ is the solution to \eqref{eq:Schr-initial-problem} with initial data $u_0$ satisfying the regularity conditions above. Using a limiting argument, the authors study the dispersion of solutions with periodic initial data, and in particular when the initial data is the Dirac comb
\begin{equation}
    \label{eq:Dirac-comb-def}
    D(x)= \sum_{n\in \Z} \delta(x-n),
\end{equation}
that is, the sum of Dirac deltas at integer points. They show that, for the Dirac comb, the dispersion has a periodic term, called the periodic regular limit, that has fluctuations concentrating at rational times. More precisely, this term is the tempered distribution
\begin{equation}
	\label{eq:hd-Schrodinger}
		\mathsf h_{b,\text{per}}(2t) = -\: \omega_b \: \zeta(2+2b) \sum_{\substack{ (p,q)=1\\q>0}} \frac{a_{b}(q)}{q^{2(1+b)}} \delta\Big(t-\frac{ p}{q} \Big),
\end{equation}
where $\zeta(z)$ is the Riemann zeta function; the coefficients $a(q)$ are given by $a_b(q) = 1$ if $q$ is odd, $a_b(q)=2-2^{2+2b}$ if $q\equiv 2 \, ( \text{mod}\ 4)$, and $a_b(q)=2^{2+2b}$ if $4|q$; and $\omega_b$ is given by
\begin{equation}
	\label{eq:omega-b-def}
		\omega_b = \frac{2}{(2\pi)^{2b}} \frac{\Gamma ( 2 b)}{\Gamma(b)\: |\Gamma(-b)|},
\end{equation}
with $\Gamma(z)$ the Gamma function (see Figure \ref{fig:PbD}). 
The nature of this object, with regard to its periodicity and the concentration of the fluctuations at rational time, is deeply related to the Talbot effect. In \cite{MR4469234} it was shown that $\mathsf{h}_{b,\text{per}}$ has multifractality properties, and in \cite{MR4775004} that it satisfies the Frisch--Parisi formalism.

\begin{figure}[ht]
    \centering
    \includegraphics[width=1\textwidth]{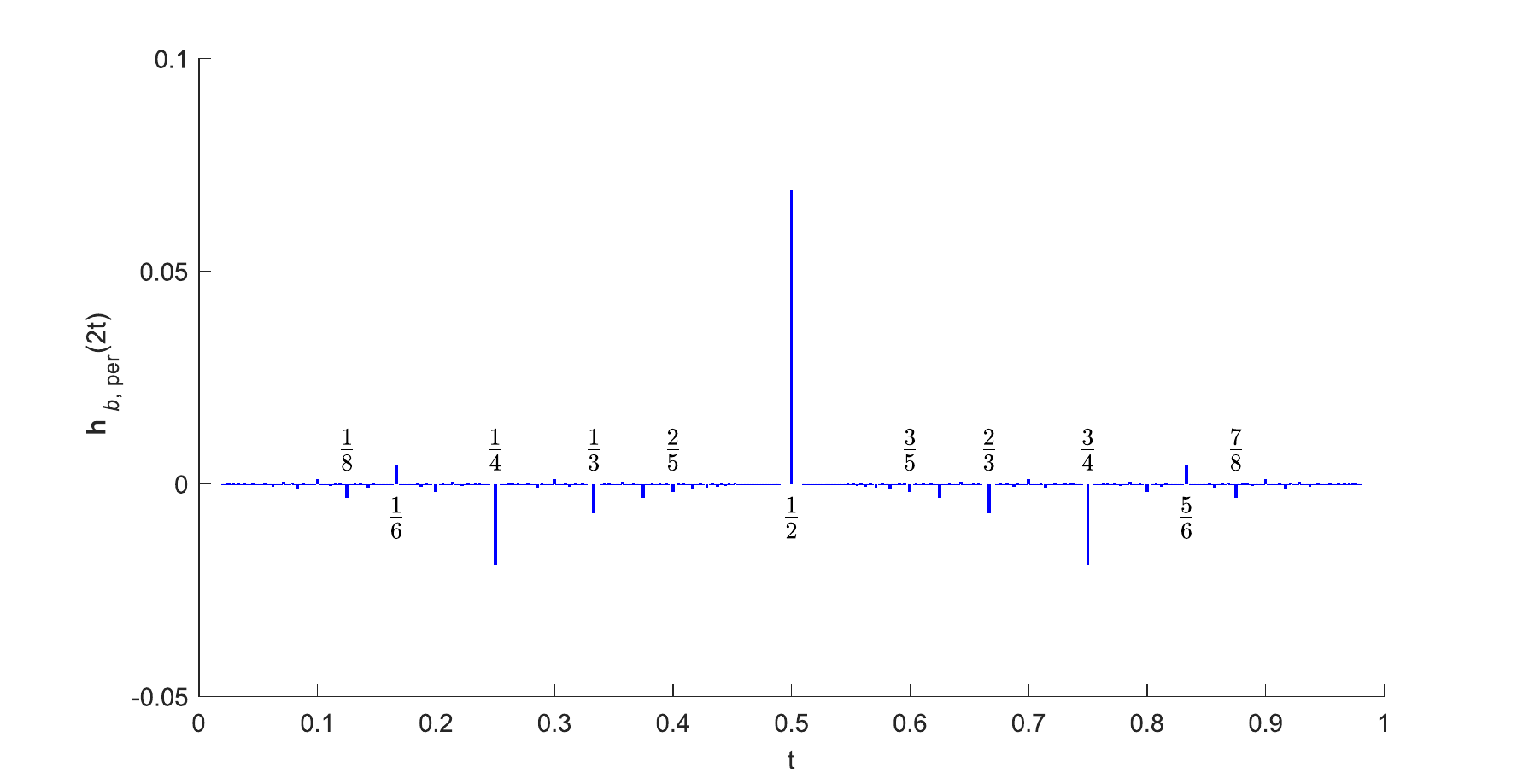}
    \caption{Plot of the distribution $\mathsf h_{b,\text{per}}(2t)$ for $b=0.25$. The Dirac deltas at rational points appear as segments of lengths proportional to their coefficients.}
    \label{fig:PbD}
\end{figure}

In this paper, we study the corresponding phenomenon for the Helmoltz equation. Suppose that $u$ is a solution to the problem
\begin{equation}
    \label{eq:Hem-eq-basic}
    \begin{cases}
        \Delta u+ u=0, & (x,y)\in \R^2\\
        u(x,0)=u_0(x),& x\in\R.
    \end{cases}
\end{equation}
We interpret $\hat u_0=F$ as the scattering data. Our goal is to take the scattering data as an approximation of the Dirac comb. By standard arguments, we see that $F$ needs to be compactly supported in order to be admissible as scattering data. For our purposes, it is convenient to rescale our problem and study
\begin{equation}
    \label{eq:Hem-eq-with-k}
    \begin{cases}
        \Delta u+4\pi^2 k^2 u=0, & (x,y)\in \R^2\\
        u(x,0)=u_0(x),& x\in\R.
    \end{cases}
\end{equation}
Now, $F=\hat u_0$ is supported on the interval $(-k,k)$. We will eventually let $k$ tend to infinity. For a solution $u$ to  \eqref{eq:Hem-eq-with-k}, we define the dispersion at height $y$ as
\begin{equation}
    \label{eq:hb-def-Helmholtz}
    h_b[F](y)=\int_{\R} |x|^{2b}|u(x,y)|^2dx.
\end{equation}

Our first main result, the Asymptotic Fractional Uncertainty Principle in Theorem~\ref{thm:dynamical-fractional-UP}, gives the asymptotic behavior of \eqref{eq:hb-def-Helmholtz}, stating that
\[
\lim_{y\rightarrow \infty}|y|^{-2b} \ h_b[F](y) \ h_b[F](0) = \omega_b \left(\int_{-k}^k \frac{|\xi|^{2b}\: |F(\xi)|^2}{(k^2-\xi^2)^b}d\xi\right) \left(\int_\R |x|^{2b}|u_0(x)|^2dx\right),
\]
with $\omega_b$ given in \eqref{eq:omega-b-def}, when the scattering data $\hat u_0=F\in L^2(-k,k)$ satisfies the conditions 
\begin{equation}
    \label{eq:Hem-regularity-conditions}
\int_{\R}|x|^{2b}|u_0(x)|^2dx<\infty, \qquad \int_{(-k,k)}\frac{|\xi|^{2b}|F(\xi)|^2}{k^2-\xi^2}d\xi<\infty.
\end{equation} 

Next, we consider periodic scattering data $u_0(x+2\pi)=u_0(x)$. Since $F=\hat u_0$ has to be supported on the interval $(-k,k)$, we can write
\begin{equation}
    \label{eq:Truncated--Dirac-Comb-def}
    F(\xi)=\sum_{|n|<k}F(n) \delta(\xi-n).
\end{equation}
We approximate these periodic scattering data with Schwartz functions---see Section~\ref{seq:fractional} for details.
In Theorem~\ref{thm:hb-periodic-decomposition}, we identify two terms in their dispersion: a singular term $S_b^\varepsilon$ and a regular term $R_b^\varepsilon$. The singular term blows up as $\varepsilon$ tends to zero. On the contrary, the regular term, that encodes the interaction among different frequencies, tends to a very interesting function $R_b$ given by
    \begin{equation}
        \label{eq:periodic-limit-function}
        R_b[F](y) = -2\omega_b \sum_{\substack{|n|,|m|<k \\ n\neq m}} \frac{F(n)\overline{ F(m)}}{|n-m|^{1+2b}} e^{2\pi i( \sqrt{k^2-n^2}-\sqrt{k^2-m^2})y},
    \end{equation}
where $\omega_b$ is given again by \eqref{eq:omega-b-def}.

Finally, we consider the Dirac comb $D(x)=\sum_{n\in \Z}\delta(x-n)$, as a limiting case of \eqref{eq:Hem-eq-with-k} when the scaling $k$ tends to infinity. In our second main result, Theorem~\ref{thm:conv-limit-k}, we show that, under the correct rescaling, the regular part \eqref{eq:periodic-limit-function} of the dispersion tends to $\mathsf h_{b,\text{per}}$ in \eqref{eq:hd-Schrodinger}, i.e., the periodic part of the dispersion for the solution to the Schr\"odinger equation with the Dirac comb as initial data, obtained in \cite{MR4469234}. This convergence occurs in the Sobolev space $H^{s}$ with $s<-\tfrac{1}{2}.$


\subsection{Notation and preliminaries}
\label{seq:prelimiraies}

Throughout this article, $\hat{f}$ and $\F(f)$ denote the Fourier transform of $f$, that is
\begin{equation*}
	\hat{f}(\xi) =  \int_{\R} e^{-2\pi ix\xi} f(x) dx.
\end{equation*}
Similarly $\check{f}$ and $\F^{-1}(f)$ denote the inverse Fourier transform of $f$.

We write $A\lesssim B$ if there exists a constant $c>0$ such that $A\leq cB$, and $A\simeq B$ if both $A\lesssim B$ and $B\lesssim A$ hold. By $A\lesssim_xB$, we mean that the implicit constant depends on $x$.

Finally, we introduce some standard tools in fractional analysis.
\begin{definition}[Fractional derivatives, \cite{NahasPonce2009}]\label{def:fracional-derivatives}
	Let $0<b<1$. We define the fractional derivatives
\begin{equation}
	\label{eq:fractional-derivative-def}
	\begin{aligned}
		\mathcal D^{b} f(x)=
		\left(\int_{\mathbb{R}} \frac{|f(x)-f(y)|^{2}}{|x-y|^{1+2 b}} d y\right)^{\frac{1}{2}},
		\qquad 
		D^{b} f(x)=
		\mathcal{F}^{-1} \big( |\xi|^{b} \hat{f}\big)(x).
	\end{aligned}
\end{equation}
\end{definition}
These are the standard definitions of fractional derivatives, using singular integrals and the Fourier transform, respectively. They can be employed to define equivalent non-homogeneous Sobolev spaces.

\begin{theorem}[Stein \cite{Stein1961,Stein1970-b}]
	Let $0<b<1$ and $f\in L^2(\R)$, then
	\begin{equation}\label{eq:equivalence-fractional-norms}
			\|f\|_{2}+\|D^{b} f \|_{2} \simeq\|f\|_{2}+  \|\mathcal D^{b} f \|_{2} \ .
	\end{equation}
\end{theorem}

\begin{proposition}[Fractional Leibniz rule, Proposition 1 in \cite{NahasPonce2009}]
    Let $0<b<1$. If $f,g$ are measurable functions, then
	\begin{equation}\label{eq:leibniz-rule-fractional} 
		\|\mathcal D^{b}(fg)\|_{2} \leqslant\|f \mathcal D^{b}g\|_{2}+\|g\mathcal  D^{b} f\|_{2} \ .
	\end{equation}
\end{proposition}
We recall the following relation between the two fractional derivatives, used throughout our discussion, see \cite{Frank2018}. For $f\in H^b(\R)$, with $0<b<1$, we have
\begin{equation}
    \label{eq:representation-hb-Fractional}
        \int_{\R}|x|^{2b} |\hat f(x)|^2dx = \omega_b \int_\R \int_\R \frac{| f(\xi) - f(\eta)|^2}{|\xi-\eta|^{1+2b}}d\xi d\eta,
    \end{equation}
where $ \omega_b $ is given by \eqref{eq:omega-b-def}.


\section{Non-fractional analysis}
\label{seq:non-fractional}

As mentioned in the previous section, we will work with the scale parameter $k>0$. Let $u$ be a solution to the Helmholtz equation in the plane,
\begin{equation}
	\Delta u+\label{eq:Hemholtz-eq-k}4\pi^2 k^{2} u=0.
\end{equation}
We will assume that $u$ is locally in $L^2$ and that satisfies 
$
\sup \frac{1}{R}\int_{B_R} |u|^2<\infty,
$
where the supremum is taken over all the balls $B_R$ of radius $R$. By the Agmon--H\"ormander condition \cite{AgmonHormander1976}, we have that the Fourier transform of $u$ is an $L^2$ function $f$ on the circle $\mathbb S^1_k$ of radius $k$. That is, we can write
\begin{align*}
	u(x, y)&=
	 \int_{\S^{1}_k} f(\xi, \eta) e^{2 \pi i(x\xi+y \eta)} d \sigma(\xi, \eta) 
	\\
    &=  
	\int_{-k}^{k} f\left(\xi, \sqrt{k^{2} - \xi^{2}}\right) e^{2 \pi i\left(x\xi+y \sqrt{k^{2}-\xi^{2}}\right)} \frac{k\: d \xi}{\sqrt{k^{2}-\xi^{2}}}  \\
    &
 \qquad +\int_{-k}^{k} f\left(\xi,-\sqrt{k^{2}-\xi^{2}}\right) e^{2 \pi i\left(x\xi-y \sqrt{k^{2} - \xi^{2}}\right)} \frac{k\: d \xi}{\sqrt{k^{2}-\xi^{2}}}.
\end{align*}
Here, the two terms correspond to the upper and lower halves of the circle. Note that the weight $(k^2-\xi^2)^{-\frac 12}$ naturally splits the integral into the two parts. Each of these parts will give rise to two solutions to the equation \eqref{eq:Hemholtz-eq-k}, $u^{+}$ and $u^{-}$. For simplicity, we shall work only with the upper half. That is, our solution to the Helmholtz equation \eqref{eq:Hemholtz-eq-k} can be explicitly written as
\begin{equation}
    \label{eq:solution-explicit-formula}
    u(x,y)
    =\int_{-k}^k F(\xi) e^{2 \pi i\left(x\xi+y \sqrt{k^{2}-\xi^{2}}\right)}  d\xi,
\end{equation}
for an $L^2$-function $F$ supported on $(-k,k)$. Note that, from \eqref{eq:solution-explicit-formula}, $F$ is the Fourier transform of $u_0(x)=u(x,0)$. We interpret $F$ as the scattering data as $y\rightarrow \infty$, since, by \eqref{eq:solution-explicit-formula} and the stationary phase method, we obtain the asymptotic behavior
\begin{equation}
\label{eq:scattering-infinity}
u(x,y) =   \frac{ k}{\sqrt y}  \: F\Big(\frac{kx}{y}\Big)\:e^{2\pi i k y-\frac{i\pi}{4}} + O(y^{-\frac 32}), \quad |x|\ll y.
\end{equation}

\subsection{Dispersion}

Let $u$ be a solution of \eqref{eq:Hemholtz-eq-k} with scattering data $F=\hat u_0\in L^2(-k,k)$. We define the dispersion of the solution at height $y\in \R$ as
\begin{equation}
 	\label{eq:h1-def}
    h_{1}[F](y) := \int_{\mathbb{R}}|x u(x, y)|^{2} d x.
\end{equation}
Here, we assume that the dispersion at $y=0$ is finite, i.e., we let $h_1[F](0)=\int_\R|xu_0(x)|^2dx<\infty$. If there is no risk for confusion, we may omit the $F$ in the notation and simply write $h_1$.

Integration by parts and the Fourier inversion formula can be used to derive an explicit expression for this dispersion in terms of the scattering data. Moreover, it is straightforward to verify that the formula is precisely a second-order polynomial.  
 \begin{lemma}\label{lem:h1-expression}
     Let $u$ be a solution to \eqref{eq:Hem-eq-with-k} with scattering data $F=\hat u_0\in L^2(-k,k)$ satisfying the regularity conditions
     \begin{equation}
     	\label{eq:regularity-hypotheses}
         \int_{\mathbb{R}}|x u_0(x)|^{2} d x < +\infty, \qquad \int_{-k}^{k} \frac{|\xi F(\xi)|^{2}}{k^{2}-\xi^{2}} d \xi < +\infty.
     \end{equation}
     Then, $h_1(y)$ is a second-order polynomial in $y$. More precisely, for $y\in \R$ we have
     \begin{equation}\label{eq:h1}
         h_1(y)= \int_{\R} |x u(x,0)|^2 dx + \frac y\pi \int_{-k}^{k}\im \big(\overline{F'(\xi)}F(\xi)\big)  \frac{\xi\:  d\xi}{\sqrt{k^2-\xi^2}}
	+
        {y^{2}}
        \int_{-k}^{k} \frac{|\xi F(\xi)|^{2}}{k^{2}-\xi^{2}} d \xi .
     \end{equation} 
\end{lemma}

\begin{demostraciones-fuera}

\begin{proof} Let $F(\xi)=\hat u_0(\xi)$ as in \eqref{eq:F-definition}. Using \eqref{eq:solution-explicit-formula}, we can explicitly compute , integrating by parts,
\begin{equation*}
	\begin{aligned}
		x u(x, y)
		&=
		\int_{-k}^{k} x\: F(\xi) e^{2 \pi i x\xi} e^{2 \pi i y \sqrt{k^{2}-\xi^{2}}} d \xi
		\\&=
		\frac{1}{2\pi i} \int_{-k}^{k} e^{2 \pi i x\xi} \left(F^{\prime}(\xi)-\frac{2 \pi i y \xi F(\xi)}{\sqrt{k^{2}-\xi^{2}}} \right) e^{2 \pi i y \sqrt{k^{2}-\xi^{2}}} d \xi .
	\end{aligned}
\end{equation*}
From here, multiply with its conjugate, we have that $|xu(x,y)|^2$ is actually equal to
\begin{equation*}
	\begin{aligned}
		\frac{1}{4\pi^2}\int_{-k}^{k} \int_{-k}^{k}  e^{2 \pi i x(\xi-\eta)}\left(F^{\prime}(\xi)-\frac{2 \pi i y \xi}{\sqrt{k^{2}-\xi^{2}}}\right) 
    \left(\overline{F^{\prime}(\eta )}+\frac{2 \pi i y \eta}{\sqrt{k^{2}-\eta^{2}}}\right) e^{2 \pi i y\left(\sqrt{k^{2}-\xi^{2}}-\sqrt{k^{2}-\eta^{2}}\right)} d\xi d\eta.
	\end{aligned}
\end{equation*}
So, by the Fourier inversion formula, we have
$$
\begin{aligned}
	h_1(y) & = \int_{\R} |xu(x,y)|^2 dy \\
	&=	\frac{1}{4\pi^2} \int_{-k}^{k} \left|F^{\prime}(\xi)-\frac{2 \pi i y \xi F(\xi)}{\sqrt{k^{2}-\xi^{2}}}\right|^{2} d \xi \\
    & =\frac{1}{4\pi^2}\int_{-k}^{k}\left|F^{\prime}(\xi)\right|^{2} d \xi+\frac{1}{\pi} \re \int_{-k}^{k}-\overline{F^{\prime}(\xi)} \frac{i y \xi F(\xi)}{\sqrt{k^{2}-\xi^{2}}} d \xi
	+\frac{1}{4\pi^2}\int_{-k}^{k}\left|\frac{y {\xi} F(\xi)}{\sqrt{k^{2}-\xi^{2}}}\right|^{2} d \xi
	\\
    &= \int_{\R} |x u(x,0)|^2 dx + \frac y\pi \int_{-k}^{k}\im \big(\overline{F'(\xi)}F(\xi)\big)  \frac{\xi\:  d\xi}{\sqrt{k^2-\xi^2}}
	+\frac{y^{2}}{4\pi^2} \int_{-k}^{k} \frac{|\xi F(\xi)|^{2}}{k^{2}-\xi^{2}} d \xi ,
\end{aligned}
$$
which is the second order polynomial we where searching for.
\end{proof}

\begin{remark}
	Note that, for the three coefficients to be finite, we just need
	\begin{enumerate}[label=(\roman*)]
		\item $\int_{\mathbb{R}}|x u(x, 0)|^{2} d x < +\infty$,
		\item $\int_{-k}^{k} \frac{|\xi F(\xi)|^{2}}{k^{2}-\xi^{2}} d \xi < +\infty$,
	\end{enumerate}
	whereas the finiteness of the coefficient of first order is implied by the above two by Cauchy-Schwartz inequality.

 Also, note that if $F$ is real, then the dispersion has its minimum at $y=0$.
\end{remark}
\end{demostraciones-fuera}

\subsection{Periodic scattering data for \texorpdfstring{$h_1$}{h1}}

We want to study the dispersion when the scattering data $u_0$ is $2\pi$-periodic. In this case, $F=\hat u_0$ can be written as \eqref{eq:Truncated--Dirac-Comb-def}. Since $F$ is not an $L^2$ function, we cannot directly apply Lemma \ref{lem:h1-expression}. We will approximate $F$ as follows. Let $\varphi \in \mathscr S(\R)$ be real, non-negative, with $\|\varphi\|_2 =1$ and $\supp  \varphi\subset [-\frac 12, \frac 12]$, and let
\begin{equation}\label{eq:initial-data-periodic-1-eps}
 F^{\varepsilon} (\xi) =\frac 1{ \sqrt{\varepsilon}}  \sum_{|n|<k}  F(n)  \varphi\Big(\frac{\xi- n}{\varepsilon}\Big).
\end{equation}
Note that, in this case, $ \varphi(\frac{\xi-n}{\varepsilon})$ and $ \varphi(\frac{\xi-m}{\varepsilon})$ are disjointedly supported if $m\neq n$ and $\varepsilon \leq  1$. Additionally, the chosen scaling preserves the $L^2$-norm of $F^\varepsilon$. Let us also denote $c_\varphi= \|\partial_\xi\varphi\|_{2}^2$.

We will study the dispersion of the solution $u_\varepsilon$ of \eqref{eq:Hem-eq-with-k} with scattering data given by  \eqref{eq:initial-data-periodic-1-eps}. We identify the behavior of $h_1[F^\varepsilon]$ as $\varepsilon$ tends to zero with two terms: a blow-up term that is constant in $y$, and a stable term of order $y^2$.

\begin{theorem}\label{thm:h1-periodic}
 Let $u_\varepsilon$ be the solution to \eqref{eq:Hemholtz-eq-k}, with scattering data $F^\varepsilon$ given by $\eqref{eq:initial-data-periodic-1-eps}$. Then
\begin{equation}
	\label{eq:h1-periodic-bahaviour}
	\lim_{\varepsilon \rightarrow 0} \Big(h_1[F^\varepsilon](y) - \frac { c_\varphi}{4\:\pi^2\:\varepsilon^2}\sum_{|n|<k} |F(n)|^2\Big) = {y^2} \sum_{|n|<k}|F(n)|^2 \frac{n^2}{k^2-n^2}.
	\end{equation}
\end{theorem}

\begin{proof}
Let us compute $h_1(y)$. By \eqref{eq:h1} we have
\begin{demostraciones-fuera}
\begin{align*}
    h_1(y) &= \frac{1}{4\pi^2}
    \int_{-k}^k |\partial_\xi F^\varepsilon(\xi)|^2 d\xi 
    + \frac{y}{\pi} \im \int_{-k}^{k} \big(\overline{\partial_\xi F^\varepsilon(\xi)}F^\varepsilon(\xi)\big) \frac{\xi \: d\xi }{\sqrt{k^2-\xi^2}} 
    +\frac{y^{2}}{4\pi^2} \int_{-k}^{k} \frac{|\xi F^\varepsilon(\xi)|^{2}}{k^{2}-\xi^{2}} d \xi  \\
    & = \frac{A}{4\pi^2}+\frac{B}{\pi} \: y + \frac{C}{4\pi^2}\: y^2
\end{align*}
 \end{demostraciones-fuera}
\begin{demostraciones-dentro}
\begin{equation}
	\label{eq:h1-periodic-expansion}	
    h_1(y) = \frac{1}{4\pi^2}
    \int_{-k}^k |\partial_\xi F^\varepsilon(\xi)|^2 d\xi 
    + \frac{y}{\pi} \im \int_{-k}^{k} \overline{\partial_\xi F^\varepsilon(\xi)}F^\varepsilon(\xi) \frac{\xi \: d\xi }{\sqrt{k^2-\xi^2}} 
    +y^{2} \int_{-k}^{k} \frac{|\xi F^\varepsilon(\xi)|^{2}}{k^{2}-\xi^{2}} d \xi.
\end{equation}
\end{demostraciones-dentro}
We compute the first term. By the disjointedness of the supports, we have that
\begin{demostraciones-fuera}
\begin{align*}
    A=\int_{\R} |\partial_\xi F^\varepsilon(\xi)|^2 d\xi 
    & = \int_{-k}^k \left| \sum_{|n|<k } F(n) \frac{1}{\sqrt{\varepsilon}}\hat \varphi_\xi \Big(\frac{\xi-n}{\varepsilon}\Big) \frac 1\varepsilon \right|^2 d\xi \\
    & =  \frac{1}{\varepsilon^2} \sum_{|n|<k }|F(n)|^2 \int_{n - \frac \varepsilon{2}}^{n  + \frac\varepsilon2}\Big| \partial_\xi  \varphi \Big(\frac{\xi-n}{\varepsilon}\Big) \Big|\frac{ d\xi}{\varepsilon} \\
    & =  \frac{1}{\varepsilon^2} \sum_{|n|<k }|F(n)|^2 \int_{\R} | \hat \varphi_\xi (\xi)|^2 d\xi
\end{align*}
\end{demostraciones-fuera}
\begin{demostraciones-dentro}
\begin{align*}
    \int_{\R} |\partial_\xi F^\varepsilon(\xi)|^2 d\xi 
    %
    %
    =  \frac{1}{\varepsilon^2} \sum_{|n|<k }|F(n)|^2 \int_{n - \frac \varepsilon{2}}^{n  + \frac\varepsilon2}\Big|\partial_\xi   \varphi\Big(\frac{\xi-n}{\varepsilon}\Big) \Big|^2 \frac{ d\xi}{\varepsilon} 
    =  \frac{c_\varphi}{\varepsilon^2} \sum_{|n|<k }|F(n)|^2 .
\end{align*}
\end{demostraciones-dentro}
The following bound holds for the third term in \eqref{eq:h1-periodic-expansion}. Once again, using the disjointedness of the supports, and the fact that $\|\varphi\|_2=1$, we get
\begin{demostraciones-fuera}
\begin{align*}
    C=\int_{\R} \frac{|\xi F^\varepsilon(\xi)|^{2}}{k^{2}-\xi^{2}} d \xi dx 
    & = \int_{-k}^k \left| \sum_{|n|<k } F(n) \frac{1}{\sqrt{\varepsilon}}\hat \varphi \Big(\frac{\xi-n}{\varepsilon}\Big)  \right|^2  \frac{\xi^2}{k^2-\xi^2 }d\xi \\
    & =  \sum_{|n|<k }|F(n)|^2 \int_{n  - \frac \varepsilon{2}}^{n  + \frac\varepsilon2} \left|   \hat \varphi \Big(\frac{\xi-n}{\varepsilon}\Big)  \right|^2   \frac{\xi^2}{k^2-\xi^2 } \frac{d\xi}{\varepsilon} \\
    & \leq  \sum_{|n|<k }|F(n)|^2  \frac{(|n|+\frac \varepsilon2)^2}{k^2-(|n|-\frac{\varepsilon}{2})^2 } \int_{n  - \frac \varepsilon{2}}^{n  + \frac\varepsilon2} \left|   \hat \varphi \Big(\frac{\xi-n}{\varepsilon}\Big)  \right|^2   \frac{d\xi}{\varepsilon} \\
    & = \sum_{|n|<k }|F(n)|^2 \frac{(|n|+\frac \varepsilon2)^2}{k^2-(|n|-\frac{\varepsilon}{2})^2 }.
\end{align*}
\end{demostraciones-fuera}
\begin{demostraciones-dentro}
\begin{align*}
    \int_{\R} \frac{|\xi F^\varepsilon(\xi)|^{2}}{k^{2}-\xi^{2}} d \xi dx 
    %
    %
     =  \sum_{|n|<k }|F(n)|^2 \int_{n  - \frac \varepsilon{2}}^{n  + \frac\varepsilon2} \Big|    \varphi \Big(\frac{\xi-n}{\varepsilon}\Big)  \Big|^2   \frac{\xi^2}{k^2-\xi^2 } \frac{d\xi}{\varepsilon} 
     %
    %
    \leq  \sum_{|n|<k }|F(n)|^2 \frac{(|n|+\frac \varepsilon2)^2}{k^2-(|n|+\frac{\varepsilon}{2})^2 }.
\end{align*}
\end{demostraciones-dentro}
Similarly, we can prove the converse inequality
    \[
        \int_{\R} \frac{|\xi F^\varepsilon(\xi)|^{2}}{k^{2}-\xi^{2}} d \xi dx  \geq \sum_{|n |<k }|F(n)|^2 \frac{(|n|-\frac{ \varepsilon}{2})^2}{ k^2-(|n|-\frac{ \varepsilon}{2})^2 } .
    \]
Finally, we study the second term in \eqref{eq:h1-periodic-expansion}. Again by disjointness of the supports, we have
\begin{demostraciones-fuera}
\begin{align*}
    B & = \im \int_{-k}^{k} \frac{\overline{\partial_\xi F^\varepsilon(\xi)}F^\varepsilon(\xi)\xi}{\sqrt{k^2-\xi^2}} d\xi \\
    & = \im \int_{-k}^{k} 
        \left( \frac 1 {\sqrt{\varepsilon}}\sum_{|n |<k} \overline{F(n)} \partial_\xi \hat \varphi\Big(\frac{\xi-n}{\varepsilon} \Big) \frac 1\varepsilon\right) 
        \left( \frac 1{\sqrt{\varepsilon}} \sum_{|n |<k} F(n) \hat \varphi\Big(\frac{\xi-n}{\varepsilon} \Big) \right) 
        \frac{\xi d\xi}{\sqrt{k^2-\xi^2}} \\
    & =\frac 1\varepsilon  \sum_{|n |<k} |F(n)|^2 
        \im \int_{n\  - \frac \varepsilon{2}}^{n  + \frac\varepsilon2} 
        \partial_\xi  \varphi\Big(\frac{\xi-n}{\varepsilon} \Big) \varphi\Big(\frac{\xi-n}{\varepsilon} \Big) 
        \frac{\xi}{\sqrt{k^2-\xi^2}} \frac{d\xi}{\varepsilon} \\
        &=0
\end{align*}
\end{demostraciones-fuera}
\begin{demostraciones-dentro}
\begin{align*}
    \im \int_{-k}^{k} \frac{\overline{\partial _\xi F^\varepsilon(\xi)}F^\varepsilon(\xi)\xi}{\sqrt{k^2-\xi^2}} d\xi 
    & =\frac 1\varepsilon  \sum_{|n |<k} |F(n)|^2 
        \im \int_{n\  - \frac \varepsilon{2}}^{n  + \frac\varepsilon2} 
        \partial_\xi  \varphi\Big(\frac{\xi-n}{\varepsilon} \Big) \varphi\Big(\frac{\xi-n}{\varepsilon} \Big) 
        \frac{\xi}{\sqrt{k^2-\xi^2}} \frac{d\xi}{\varepsilon} 
    =0,
\end{align*}
\end{demostraciones-dentro}
since the integrand is real. This proves the formula \eqref{eq:h1-periodic-bahaviour}.
\end{proof}

Theorem \ref{thm:h1-periodic} identifies two terms in the dispersion of solutions with sums of deltas. 
The first term, which blows up as $\varepsilon \rightarrow 0$ and is constant in $y$, is related to the $\ell^2$-norm of the scattering data $F$. The second term, constant in $\varepsilon$ but encoding the uncertainty with respect to $y$, depends on $u_0$, determined by the appropriate weight in the $\ell^2$-norm of $F$. Furthermore, there is no true interaction among the different frequencies of the scattering. In the following sections, we will show that this is not the case when we study the fractional dispersion.


\section{Fractional dispersion}
\label{seq:fractional}

Let $u$ be a solution to the Hemholtz equation \eqref{eq:Hemholtz-eq-k}, with scattering data $F=\hat u_0 \in L^2(-k,k)$. We define the dispersion of order $b\in(0,1)$ of $u$ at height $y$ as
\begin{equation}\label{eq:fractional-dispersion-def}
	h_{b}[F](y) := \int_{\mathbb{R}}|x|^{2b} \:|u(x, y)|^{2} d x.
\end{equation}
If there is no ambiguity, we will drop the $F$ and simply write $h_b$. We are assuming that $h_b(0)=\int_{\R}|x|^{2b}|u_0(x)|^2dx <\infty$. Similarly as in the non-fractional case, we have that the dispersion grows as a power of $y$, under some conditions on $u_0$ and $F$.

\begin{theorem}[Asymptotic fractional uncertainty principle]\label{thm:dynamical-fractional-UP}
Let $u$ be a solution to the Hemholtz equation \eqref{eq:Hem-eq-with-k}, with scattering data $F=\hat u_0\in L^2(-k,k)$ satisfying the conditions
\begin{equation}
	\label{eq:regularity-conditions-fractional}
	\int_{\R}|x|^{2b}\: |u_0(x)|^2dx<\infty, \qquad \int_{-k}^k \frac{|\xi|^{2b}\: |F(\xi)|^2}{k^2-\xi^2}d\xi<\infty.
\end{equation}
Then, recalling the definition of $\omega_b$ given in \eqref{eq:omega-b-def}, we have that
\[
\lim_{y\rightarrow \infty}|y|^{-2b} \ h_b(y) \ h_b(0) = \omega_b \left(\int_{-k}^k \frac{|\xi|^{2b}\: |F(\xi)|^2}{(k^2-\xi^2)^b}d\xi\right) \left(\int_\R |x|^{2b}|u_0(x)|^2dx\right).
\]
\end{theorem}
The rest of this section is dedicated to proving Theorem~\ref{thm:dynamical-fractional-UP}, and is organized as follows. First, we derive the upper bound of $h_b(y)$ in Theorem~\ref{th:upper-bound-h-delta}. Next, we present a representation of $\hat h_b$ as a tempered distribution in Theorem~\ref{th:hb-transform-formula}. Finally, we analyze the behavior of $\hat h_b$ around zero in Proposition~\ref{prop:singularity-h-hat}, completing the proof of Theorem~\ref{thm:dynamical-fractional-UP}.
\begin{theorem}\label{th:upper-bound-h-delta}
    Let $u$ be a solution to the Helmholtz equation \eqref{eq:Hemholtz-eq-k}, with scattering data $F=\hat u_0\in L^2(-k,k)$ satisfying \eqref{eq:regularity-conditions-fractional}. Then, for $y\in \R$,
    \begin{equation}\label{eq:hdelta-upper-bound}
    h_b(y) \lesssim (1+|y|^{b})\|F\|^2_2 + \| |x|^b u_0(x) \|_2^2 
    +k^2(1+|y|^{b})
        \Big\| \frac{|\xi|^bF(\xi) }{\sqrt{k^2-\xi^2}} \Big\|^2_2,  
    \end{equation}
    where the implicit constant depends only on $b$.
\end{theorem}

\begin{comentario}
In order to prove \eqref{eq:hdelta-upper-bound}, we note we can write 
$h_b(y)$  is actually the Fourier-fractional derivative of the Fourier transform of $u$ at time $y$. So, using the norm equivalence \eqref{eq:equivalence-fractional-norms} and the Leibniz rule \eqref{eq:leibniz-rule-fractional}, we have
\begin{demostraciones-fuera}
\begin{equation*}
	\begin{aligned}
		h_{b}(y)
		&=\int_{\R} ||x|^{b} u(x, y)|^{2} d y
		\\
		&=\| D^{b}\left(\hat{u}(\cdot, y) \right)  \|^2_{2}
		\\
		&\leq \|\hat{u}(\cdot, y)\|^2_{2}+\left\|D^{b}(\hat{u}(\cdot, y))\right\|^2_{2}
		\\
        &\simeq \|\hat{u}(\cdot, y)\|^2_{2}+\left\| \mathcal D^{b}(\hat{u}(\cdot, y))\right\|^2_{2}
		\\
		& = \|\hat{u}( \cdot , 0)\|^2_{2}+\left\| \mathcal D^{b} \left(F(\xi) e^{2 \pi i y \sqrt{k^2-\xi^{2}}}\right) \right\|^2_{2} 
		\\
		&  \leq\|\hat{u}( \cdot , 0)\|^2_{2}+\left\|\mathcal D^{b} F(\xi)\right\|_{2}^2+\left\|F(\xi) \mathcal  D^{b} e^{2 \pi i y\sqrt{k^2-\xi^{2}}}\right\|^2_{2}
		\\
		& \simeq \|u(\cdot, 0)\|^2_{2}+\left\||x|^{b} u(x, 0)\right\|^2_{2} +\left\|F(\xi) \mathcal D^{b} \left( e^{2 \pi i y\sqrt{k^{2}-\xi^{2}}} \right) \right\|^2_{2}
	\end{aligned}
\end{equation*}
\end{demostraciones-fuera}
\begin{demostraciones-dentro}
\begin{equation*}
	\begin{aligned}
		h_{b}(y)
		& \lesssim \|u(\cdot, 0)\|^2_{2}+\Big\||x|^{b} u(x, 0)\Big\|^2_{2} +\Big\|F(\xi) \mathcal D^{b} \big( e^{2 \pi i y\sqrt{k^{2}-\xi^{2}}} \big) \Big\|^2_{2}.
	\end{aligned}
\end{equation*}
\end{demostraciones-dentro}
So, in order to obtain the upper bound \eqref{eq:hdelta-upper-bound}, we need to compute $ \mathcal D^b ( e^{2 \pi iy \sqrt{k^{2}-\xi^{2}}})$. Note that both the multiplier $e^{2 \pi iy \sqrt{k^{2}-\xi^{2}}} $ and its fractional derivative are always multiplying $F$. Since $F$ is supported in $(-k,k)$  the values of the multiplier outside that interval will not interfere in our computations, and therefore we may choose them freely. For convenience, we choose the multiplier to be
\begin{equation}\label{eq:multiplier-def}
m_y(\xi)=e^{2 \pi i y\sqrt{k^{2}-\xi^{2}}} \chi_{|\xi|\leq k}(\xi)+\chi_{|\xi|\geq k}(\xi).
\end{equation}
This choice will be important when computing its fractional derivative for small $y$.

\end{comentario}

\begin{lemma}\label{lem:fractional-computation}
	Let $y\in \R$, $b\in (0,1)$, and $m_y$ defined by
    \begin{equation}\label{eq:multiplier-def}
m_y(\xi)=e^{2 \pi i y\sqrt{k^{2}-\xi^{2}}} \chi_{|\xi|\leq k}(\xi)+\chi_{|\xi|\geq k}(\xi).
\end{equation} Then
	\begin{equation}\label{eq:Fractional-computation}
        \mathcal D^{b} m_{y}(\xi)\lesssim_{b}  \frac{(k|y|)^b}{(k^2-\xi^2)^{\frac b2
        }} + \frac{k(1+|\xi|^b\: |y|^{\frac{b}{2}})}{\sqrt{k^2-\xi^2}}, \quad \xi \in (-k,k) .
    \end{equation}

\end{lemma}

\begin{demostraciones-fuera}
\begin{proof}
	We write 
 \begin{align*}
 (\mathcal D^bm _y(\xi))^2
 & = \int_{\R} \frac{|m_y(\xi)-m_y(\xi-\eta)|^2}{|\eta|^{1+2b}}d\eta \\
 & = \int_{|\xi-\eta|\leq k} \frac{|e^{2\pi iy\sqrt{k^2-\xi^2}}-e^{2\pi iy\sqrt{k^2-(\xi-\eta)^2}}|^2}{|\eta|^{1+2b}}d\eta + \int_{|\xi-\eta|> k} \frac{|e^{2\pi iy\sqrt{k^2-\xi^2}}-1|^2}{|\eta|^{1+2b}}d\eta\\
 & = I_1 + I_2.
 \end{align*}
 %
%

We first bound $I_2$ by

 \[
     I_2
     = |e^{2\pi iy\sqrt{k^2-\xi^2}}-1|^2 \int_{|\xi-\eta|\geq k}\frac{d\eta}{|\eta|^{1+2b}} = \frac{|e^{2\pi iy\sqrt{k^2-\xi^2}}-1|^2}{2b(k^2-\xi^2)^{2b}}= \frac{1-\cos 2\pi y\sqrt{k^2-\xi^2}}{b(k^2-\xi^2)^{2b}} \leq \frac1{\pi b} \: y^{2b},
 \]
since $(1-\cos t)t^{-2b}\leq 2$ for all $t\geq 0$.
  We deal now with the integral $I_1$. By a change of variables, we have
  \begin{align*}
  I_1 &
   = \int_{|\xi-\eta|\leq k} \frac{|1-e^{2\pi iy\big( \sqrt{k^2-(\xi-\eta)^2}-\sqrt{k^2-\xi^2}\big) }|^2}{|\eta|^{1+2b}}d\eta \\
  & =\int_{|\xi-\eta|\leq k} \left|1-e^{2\pi iy\frac{2\xi \eta -\eta^2 }{\sqrt{k^2-(\xi-\eta)^2}+\sqrt{k^2-\xi^2}}}\right|^2\frac{d\eta }{|\eta|^{1+2b}} \\
  &= y^b \int_{|\xi-\eta /\sqrt{y}|\leq k} \left|1-e^{2\pi i\frac{2 \sqrt{y} \xi \eta -\eta^2 }{\sqrt{k^2-(\xi-\eta/\sqrt{y})^2}+\sqrt{k^2-\xi^2}}}\right|^2\frac{d\eta}{|\eta|^{1+2b}}
  \end{align*}
We want to split the integration domain $E=(\sqrt{y}(\xi-k),\sqrt{y}(\xi+k))$. In order to do that, we define:
\begin{align*}
    E_1 &=\{\eta\in E: |\eta|\geq \frac{1}{\sqrt{y}|\xi|} \}, \\
    E_2&=\{\eta\in E  : |\eta|\leq \min \big( \frac{1}{\sqrt{y}|\xi|},  \sqrt{y}|\xi|\big)  \},\\
    E_3 &= \{\eta\in E  :  \sqrt{y}|\xi|\leq |\eta|\leq \frac{1}{\sqrt{y}|\xi|}\}.
\end{align*}
Note that whenever $\sqrt{y}|\xi|\geq 1$, $E=E_1\cup E_2$ and $E_3$ is empty. Otherwise, $E=E_1\cup E_2\cup E_3$. The integral over $E_1$ is the easy one:
\[
\int_{E_1}\left|1-e^{2\pi i\frac{2 \sqrt{y} \xi \eta -\eta^2 }{\sqrt{k^2-(\xi-\eta/\sqrt{y})^2}+\sqrt{k^2-\xi^2}}}\right|^2\frac{d\eta}{|\eta|^{1+2b}}
\leq 4 \int_{|\eta|\geq \frac{1}{\sqrt{y}|\xi|}} \frac{d\eta}{|\eta|^{1+2b}}
= \frac{4 y^b |\xi|^{2b}}{b}.
\]
The integral over $E_2$:
\begin{align*}
    \int_{E_2}  \left|1-e^{2\pi i\frac{2 \sqrt{y} \xi \eta -\eta^2 }{\sqrt{k^2-(\xi-\eta/\sqrt{y})^2}+\sqrt{k^2-\xi^2}}}\right|^2\frac{d\eta}{|\eta|^{1+2b}}
    &  \leq \int_{E_2} \left| \frac{2\sqrt{y}\xi \eta -\eta^2 }{\sqrt{k^2-(\xi-\eta/\sqrt{y})^2}+\sqrt{k^2-\xi^2}}\right|^2 \frac{d\eta}{|\eta|^{1+2b}} \\
    & \leq \frac{1}{k^2-\xi^2} \int_{E_2} \frac{(2\sqrt{y}|\xi| |\eta|+\sqrt{y}|\xi||\eta|)^2}{|\eta|^{1+2b}}d\eta \\
    & \leq \frac{9y|\xi|^2}{k^2-\xi^2} \int_{|\eta|\leq \frac{1}{\sqrt{y}|\xi|}} \frac{d\eta}{|\eta|^{2b-1}} 
    = \frac{9y^b|\xi|^{2b}}{(1-b)(k^2-\xi^2)}.
\end{align*}
Finally, while estimating the integral over $E_3$ we may assume that $\sqrt{y}|\xi|\leq 1$:
\begin{align*}
    \int_{E_3}  \left|1-e^{2\pi i\frac{2 \sqrt{y} \xi \eta -\eta^2 }{\sqrt{k^2-(\xi-\eta/\sqrt{y})^2}+\sqrt{k^2-\xi^2}}}\right|^2\frac{d\eta}{|\eta|^{1+2b}} 
    & \leq 4\int_{1\leq |\eta|\leq \frac{1}{\sqrt{y}|\xi|}} \frac{d\eta}{|\eta|^{1+2b}}+ 9\int_{\sqrt{y}|\xi|\leq |\eta|\leq 1}\frac{d\eta}{|\eta|^{-1+2b}} \leq \frac{4}{b}+\frac{9}{1-b}.
\end{align*}
Combining all estimates together, we get
\begin{align*}
D^b(\varphi_y)(\xi)^2
& \leq \frac{y^b}{b}
+ y^b \left(\frac{4 y^b |\xi|^{2b}}{b}+\frac{9y^b|\xi|^{2b}}{(1-b)(k^2-\xi^2)}+\frac{4}{b}+\frac{9}{1-b} \right)  \\
& \lesssim_b y^{2b} + y^b + \frac{|\xi|^{2b}}{k^2-|\xi|^2}y^{2b}.\qedhere
\end{align*}
\end{proof}
\end{demostraciones-fuera}

\begin{demostraciones-dentro}
\begin{proof}
	For simplicity, we assume $y>0$. We apply the fractional derivative \eqref{eq:fractional-derivative-def} to \eqref{eq:multiplier-def}, and we split the integration domain in two parts as follows:
 \begin{align*}
 (\mathcal D^bm _y(\xi))^2
   = \Big(\int_{|\xi-\eta|\geq k}+\int_{|\xi-\eta|\leq k}\Big)\frac{|m_y(\xi)-m_y(\xi-\eta)|^2}{|\eta|^{1+2b}}d\eta
 %
%
  = I_1 + I_2.
 \end{align*}
 %
%
%
We first bound $I_1$. Note that in this case $m_y(\xi-\eta)=1$. Therefore, we have
 \begin{align*}
     I_1
     &= \int_{|\xi-\eta|\geq k}|e^{2\pi iy\sqrt{k^2-\xi^2}}-1|^2 \frac{d\eta}{|\eta|^{1+2b}} 
     %
     %
     = \frac{1-\cos \big(2\pi y\sqrt{k^2-\xi^2}\big)}{b(k^2-\xi^2)^{2b}} \big((k+\xi)^{2b}+(k-\xi)^{2b}\big) \\
     & = \frac{(2\pi y)^{2b}}{b}\frac{1-\cos \big(2\pi y\sqrt{k^2-\xi^2}\big)}{(2\pi y \sqrt{k^2-\xi^2})^{2b}}  \frac{(k+\xi)^{2b}+(k-\xi)^{2b}}{({k^2-\xi^2} )^b}\lesssim_b  \: \frac{(ky)^{2b}}{(k^2-\xi^2)^b},
 \end{align*}
since $(1-\cos t)t^{-2b}\leq 2$ for all $t\geq 0$ and $0<b<1$.  

We are left with the integral $I_2$. Making the change of variables $\eta \mapsto \eta/\sqrt{y}$, we have
  \begin{align*}
  I_2 
   = \int_{|\xi-\eta|\leq k} |1-e^{2\pi iy( \sqrt{k^2-(\xi-\eta)^2}-\sqrt{k^2-\xi^2}) }|^2\frac{d\eta }{|\eta|^{1+2b}}
  %
  %
  = y^b \int_{|\xi-\eta /\sqrt{y}|\leq k} \big|1-e^{2\pi i\gamma(\eta)}\big|^2\frac{d\eta}{|\eta|^{1+2b}},
  \end{align*}
  where $\gamma(\eta)=\frac{2 \sqrt{y} \xi \eta -\eta^2 }{\sqrt{k^2-(\xi-\eta/\sqrt{y})^2}+\sqrt{k^2-\xi^2}}$.
We  split the integration domain into three parts. Let $E=[\sqrt{y}(\xi-k),\sqrt{y}(\xi+k)]$ and define the three parts
\begin{align*}
    E_1 &=\{\eta\in E: |\eta|\geq \frac{1}{\sqrt{y}|\xi|} \}, \\
    E_2&=\{\eta\in E  : |\eta|< \min \big( \frac{1}{\sqrt{y}|\xi|},  \sqrt{y}|\xi|\big)  \},\\
    E_3 &= \{\eta\in E  :  \sqrt{y}|\xi|\leq |\eta|< \frac{1}{\sqrt{y}|\xi|}\}.
\end{align*}
Note that whenever $\sqrt{y}|\xi|\geq 1$, $E=E_1\cup E_2$ and $E_3$ is empty. In any case, $E=E_1\cup E_2\cup E_3$. First, the integral over $E_1$ is bound by
\[
\int_{E_1}\big|1-e^{2\pi i\gamma(\eta)}\big|^2\frac{d\eta}{|\eta|^{1+2b}}
\leq 4 \int_{|\eta|\geq \frac{1}{\sqrt{y}|\xi|}} \frac{d\eta}{|\eta|^{1+2b}}
\lesssim_b y^b |\xi|^{2b}.
\]
Next, we bound the integral over $E_2$ by
\begin{align*}
    \int_{E_2}  \left|1-e^{2\pi i \gamma(\eta)}\right|^2\frac{d\eta}{|\eta|^{1+2b}}
    & \lesssim \int_{E_2}|\gamma(\eta)|^2 \frac{d\eta}{|\eta|^{1+2b}} 
     \leq \frac{1}{k^2-\xi^2} \int_{E_2} \frac{(2\sqrt{y}|\xi| |\eta|+\sqrt{y}|\xi||\eta|)^2}{|\eta|^{1+2b}}d\eta \\
    & \leq \frac{9y|\xi|^2}{k^2-\xi^2} \int_{|\eta|\leq \frac{1}{\sqrt{y}|\xi|}} \frac{d\eta}{|\eta|^{2b-1}} 
    \lesssim_b \frac{y^b|\xi|^{2b}}{k^2-\xi^2}.
\end{align*}
Finally, while estimating the integral over $E_3$ we may assume that $\sqrt{y}|\xi|\leq 1$. So, arguing as we did for the estimates for $E_1$ and $E_2$, we have
\begin{align*}
    \int_{E_3}  \left|1-e^{2\pi i\gamma(\eta)}\right|^2\frac{d\eta}{|\eta|^{1+2b}} 
    & \leq 4\int_{1\leq |\eta|\leq \frac{1}{\sqrt{y}|\xi|}} \frac{d\eta}{|\eta|^{1+2b}}+ \frac{9y|\xi|^2}{k^2-\xi^2}\int_{\sqrt{y}|\xi|\leq |\eta|\leq 1}\frac{d\eta}{|\eta|^{2b-1}} \\
    & \lesssim_b 1+ y^b|\xi|^{2b}+ \frac{1+(\sqrt y|\xi|)^{2-2b}}{k^2-\xi^2} \lesssim_b y^b|\xi|^{2b}+ \frac{k^2}{k^2-\xi^2} .
\end{align*}
Combining all estimates together, we get \eqref{eq:Fractional-computation}.
\end{proof}
\end{demostraciones-dentro}

\begin{proof}[Proof of Theorem \ref{th:upper-bound-h-delta}]
Note that $h_b(y)$  is actually the Fourier-fractional derivative of the Fourier transform of $u$ at height $y$. So, using the norm equivalence \eqref{eq:equivalence-fractional-norms} and the Leibniz rule \eqref{eq:leibniz-rule-fractional}, we have
\begin{equation*}
	\begin{aligned}
		h_{b}(y)
		& \lesssim_b \|u(\cdot, 0)\|^2_{2}+\left\||x|^{b} u(x, 0)\right\|^2_{2} +\left\|F(\xi) \mathcal D^{b} \left( e^{2 \pi i y\sqrt{k^{2}-\xi^{2}}} \right) \right\|^2_{2}.
	\end{aligned}
\end{equation*}
Using the pointwise estimate in Lemma \ref{lem:fractional-computation}, we conclude the proof.    
\end{proof}

The polynomial growth of \eqref{eq:hdelta-upper-bound} indicates that $h_b$ is actually a tempered distribution. Theorem~\ref{th:hb-transform-formula} provides an explicit formula for its Fourier transform away from zero. 

\begin{theorem}\label{th:hb-transform-formula}
    Let $u$ be a solution to the Helmholtz equation with scattering data $F\in L^2(-k,k)$ and satisfying \eqref{eq:regularity-conditions-fractional}. Then $h_b$ is a tempered distribution, whose Fourier transform can be represented away from zero as 
    \[
        \hat h_b (\tau)= -2\omega_b \int_{(-k,k)^2} \hat F(\xi) \overline{\hat F(\eta)} \delta_0 \Big( \tau-(\sqrt{k^2-\xi^2}-\sqrt{k^2-\eta^2})\Big) \frac{d\xi d\eta}{|\xi-\eta|^{1+2b}},
    \]
    where $ \omega_b $ is given by \eqref{eq:omega-b-def}. This means that for $\psi \in \mathscr S(\R) $ supported away from the origin,
    \begin{equation}
        \label{eq:FT-of-hb-expression}
        \langle \hat h_b,\psi\rangle =  -2\omega _b \int_{(-k,k)^2} \hat F(\xi) \overline{\hat F(\eta)} \psi\Big( \sqrt{k^2-\xi^2}-\sqrt{k^2-\eta^2} \Big) \frac{d\xi d\eta}{|\xi-\eta|^{1+2b}}.
    \end{equation}
\end{theorem}


\begin{demostraciones-fuera}
\begin{proof} 
    If we take the Fourier transform on $x$ in  the explicit formula for the solution \eqref{eq:solution-explicit-formula}, we have that 
    \[
        \hat u(\xi,y) = F(\xi) e^{2\pi i y \sqrt{k^2-\xi^2}}.
    \]
    We use the following classical identity from\cite{Frank2018}:
    \begin{equation}
    \label{eq:representation-hb-Fractional}
        h_b(y) = \omega_b \int_\R \int_\R \frac{|F(\xi) e^{2\pi i y \sqrt{k^2-\xi^2}}-F(\eta)e^{2\pi i y \sqrt{k^2-\eta^2}}|^2}{|\xi-\eta|^{1+2b}}d\xi d\eta.
    \end{equation}
    Note that in this reference the Fourier transform is normalized in a different way and therefore the constant is different. Let $\psi \in \mathscr S(\R)$ be supported outside the interval $(-a,a)$ for some $a>0$. We may assume that $a<1$. Then, since $h_b\in \mathscr S'(\R)$ by Theorem \ref{th:upper-bound-h-delta}, we have
    \begin{align*}
        \langle \hat h_b,\psi \rangle & = \langle h_b,\hat \psi \rangle \\ 
        & = \omega_b\int_\R \hat \psi(y)\left[ \int_\R \int_\R \frac{|F(\xi) e^{2\pi it \sqrt{k^2-\xi^2}}-F(\eta)e^{2\pi it \sqrt{k^2-\eta^2}}|^2}{|\xi-\eta|^{1+2b}}d\xi d\eta \right]dt \\
        & = \omega_b \int_\R \int_\R  \Big[ \int_\R \hat \psi(y) \Big( |F(\xi)|^2 + |F(\eta)|^2 - F(\xi)\overline{F(\eta)} e^{2\pi it(\sqrt{k^2-\xi^2}-\sqrt{k^2-\eta^2})} \\
        &
        \qquad \qquad -F(\eta)\overline{F(\xi)} e^{2\pi it(\sqrt{k^2-\eta^2}-\sqrt{k^2-\xi^2})}  \Big)dt \Big] \frac{d\xi d\eta }{|\xi-\eta|^{1+2b}}\\
        & = -2\omega_b \int_{(-k,k)^2}F(\xi)\overline{F(\eta)} \psi(\sqrt{k^2-\xi^2}-\sqrt{k^2-\eta^2})\frac{d\xi d\eta }{|\xi-\eta|^{1+2b}},
    \end{align*}
    since $\psi(0)=0$. Let us show now that this represents a bounded functional on $\mathscr S(\R)$ away from 0. We bound this integral as
    \begin{align*}
        |\langle \hat h_b,\psi \rangle | & \leq C \int |F(\xi)F(\eta)| \: | \psi(\sqrt{k^2-\xi^2}-\sqrt{k^2-\eta^2})|\frac{d\xi d\eta }{|\xi-\eta|^{1+2b}} \\
        & \leq C \int_{-k}^k |F(\xi)|^2 \int_{-k}^k \frac{|\psi(\sqrt{k^2-\xi^2}-\sqrt{k^2-\eta^2})|}{|\xi - \eta |^{1+2b}} d\eta \: d\xi \\
        &
        \qquad \qquad+C \int_{-k}^k |F(\eta)|^2 \int_{-k}^k \frac{|\psi(\sqrt{k^2-\xi^2}-\sqrt{k^2-\eta^2})|}{|\xi - \eta |^{1+2b}} d\xi \: d\eta .
    \end{align*}
    We bound the first integral, and the second one is completely symmetric. Recall that $\psi$ is supported away from $(-a,a)$. We split the integral in two:
    \[
        \int_{-k}^k \left[ |F(\xi)|^2 \int_{-k}^k \frac{|\psi(\sqrt{k^2-\xi^2}-\sqrt{k^2-\eta^2})|}{|\xi - \eta |^{1+2b}} d\eta \right] \: d\xi
        = \int_{|\xi|\leq \sqrt a}+ \int_{\sqrt a \leq |\xi|\leq k} \Big[...\Big] d\xi  := A+B
    \]
    We bound first $A$: if $|\xi|\leq \sqrt a$ and $|\xi-\eta|\leq \sqrt a $, then $     |\sqrt{k^2-\xi^2}-\sqrt{k^2-\eta^2}|\leq a$. Indeed, we have
    \begin{align*}
        |\sqrt{k^2-\xi^2}-\sqrt{k^2-\eta^2}| & = \frac{|\eta^2-\xi^2|}{\sqrt{k^2-\xi^2}+\sqrt{k^2-\eta^2}} 
         \leq \frac{ |\xi-\eta|\: |\xi+\eta|}{\sqrt{k^2-a^2}} 
         \leq \frac{3 a}{\sqrt{k^2-a^2}} \leq a, 
    \end{align*}
    assuming $k^2-a^2\geq 9$, which we may assume since $a<1$ and $k$ is large. Therefore, we have
    \begin{align*}
        A & \leq \|\psi\|_\infty \int_{|\xi|\leq \sqrt a} |F(\xi)|^2 \Big[ \int_{|\xi-\eta|\geq \sqrt a} \frac{d\eta}{|\xi-\eta|^{1+2b}} \Big]\: d\xi 
        = \frac{\|\psi\|_\infty}{b\: a^{2b}} \int_{|\xi\|\leq \sqrt a} |F(\xi)|^2 d\xi.
    \end{align*}
    Now we bound $B$. First, we claim that whenever $|\xi|\geq \sqrt a $ and $|\xi-\eta|\leq \ {a^2}/({7\xi})$, then  $     |\sqrt{k^2-\xi^2}-\sqrt{k^2-\eta^2}|\leq a$. In order to prove the claim, we consider two cases. First, if $|\xi|$ is close to $k$, more precisely, suppose $k^2-\xi^2 \leq (a/3)^2$. Then, if we call $\delta = |\xi-\eta|,$ we have
    \[|k^2-\eta^2 |\leq |k^2-\xi^2|+\delta^2 + 2\delta \xi \leq \frac{a^2}{9}+\frac{a^4}{49\xi^2}+\frac{2a^2}{7}\leq \frac{4}{5} a^2, \]
    since $a<1$ and $|\xi|\geq \sqrt a.$ Therefore
    \begin{align*}
        |\sqrt{k^2-\xi^2}-\sqrt{k^2-\eta^2}| & \leq \sqrt{k^2-\xi^2}+\sqrt{k^2-\eta^2} 
         \leq \frac{a}{3}+ \frac{16}{25}a \leq a.
    \end{align*}
    For the second case, if $|\xi|$ is far away from $k$, that is, $k^2-\xi^2 \geq (a/3)^2,$ then
    \begin{align*}
        |\sqrt{k^2-\xi^2}-\sqrt{k^2-\eta^2}| & = \frac{|\xi-\eta|\: |\xi+\eta|}{\sqrt{k^2-\xi^2}+\sqrt{k^2-\eta^2}} 
         \leq \frac{\frac{a^2}{7\xi} (\xi +\frac{a^2}{7\xi})}{a/3}= \frac{3a}{7} (1+\frac{a^2}{7\xi}) \leq a,
    \end{align*}
    when once again we have used $|\xi|\geq \sqrt a$ and $a<1$. This proves the claim. So, we can estimate $B$ as
    \begin{align*}
        B &\leq  \|\psi\|_\infty \int_{|\xi|\geq \sqrt a} |F(\xi)|^2 \Big[ \int_{|\xi-\eta|\geq  a^2/(7\xi)} \frac{d\eta}{|\xi-\eta|^{1+2b}} \Big]\: d\xi 
        = c \frac{\|\psi\|_\infty}{ a^{4b}} \int_{|\xi|\leq k} |\xi|^{2b} |F(\xi)|^2 d\xi.
    \end{align*}
    Collecting all estimates, we have proven that 
    \[|\langle \hat h_b , \psi\rangle| \leq c \frac{\|\psi\|_\infty}{a^{4b}} \int_{|\xi|\leq k } (1+|\xi|^{2b}) |F(\xi)|^2d\xi,\]
    which, by \eqref{eq:hip-bounded-functional}, represents a bounded functional on $\mathscr S (\R\setminus \{0\}).$ 
\end{proof}
\end{demostraciones-fuera}


\begin{demostraciones-dentro}
\begin{proof} 
    By applying the inverse Fourier transform to the explicit formula \eqref{eq:solution-explicit-formula}, together with \eqref{eq:representation-hb-Fractional}, we have
    \begin{equation}
        h_b(y) = \omega_b \int_\R \int_\R \frac{|F(\xi) e^{2\pi i y \sqrt{k^2-\xi^2}}-F(\eta)e^{2\pi i y \sqrt{k^2-\eta^2}}|^2}{|\xi-\eta|^{1+2b}}d\xi d\eta.
    \end{equation}
    Let $\psi \in \mathscr S(\R)$ be supported outside the interval $(-a,a)$ for some $a>0$. We may assume that $a<1$. Then, as $h_b\in \mathscr S'(\R)$ by Theorem \ref{th:upper-bound-h-delta}, we have
    \begin{align*}
        \langle \hat h_b,\psi \rangle 
        & = \omega_b\int_\R \hat \psi(y)\left[ \int_\R \int_\R \frac{|F(\xi) e^{2\pi i y \sqrt{k^2-\xi^2}}-F(\eta)e^{2\pi iy \sqrt{k^2-\eta^2}}|^2}{|\xi-\eta|^{1+2b}}d\xi d\eta \right]dy \\
        & = \omega_b \int_\R \int_\R  \Big[ \int_\R \hat \psi(y) \Big( |F(\xi)|^2 + |F(\eta)|^2 - F(\xi)\overline{F(\eta)} e^{2\pi iy(\sqrt{k^2-\xi^2}-\sqrt{k^2-\eta^2})} \\
        &
        \qquad \qquad -F(\eta)\overline{F(\xi)} e^{2\pi iy(\sqrt{k^2-\eta^2}-\sqrt{k^2-\xi^2})}  \Big)dy \Big] \frac{d\xi d\eta }{|\xi-\eta|^{1+2b}}\\
        & = -2\omega_b \int_{(-k,k)^2}F(\xi)\overline{F(\eta)} \psi(\sqrt{k^2-\xi^2}-\sqrt{k^2-\eta^2})\frac{d\xi d\eta }{|\xi-\eta|^{1+2b}},
    \end{align*}
    since $ \psi(0)=0$, which is \eqref{eq:FT-of-hb-expression}. Let us show now that the above represents a bounded functional on $\mathscr S(\R)$ away from zero. We bound this integral as
    \begin{align*}
        |\langle \hat h_b,\psi \rangle | 
        & \lesssim_b \int_{-k}^k |F(\xi)|^2 \int_{-k}^k \frac{|\psi(\sqrt{k^2-\xi^2}-\sqrt{k^2-\eta^2})|}{|\xi - \eta |^{1+2b}} d\eta \: d\xi \\
        &
        \qquad \qquad+ \int_{-k}^k |F(\eta)|^2 \int_{-k}^k \frac{|\psi(\sqrt{k^2-\xi^2}-\sqrt{k^2-\eta^2})|}{|\xi - \eta |^{1+2b}} d\xi \: d\eta .
    \end{align*}
    We bound the first integral, as the second one is completely symmetric. Recall that $\psi$ is supported away from $(-a,a)$. We split the integral in two parts:
    \[
        \int_{-k}^k \left[ |F(\xi)|^2 \int_{-k}^k \frac{|\psi(\sqrt{k^2-\xi^2}-\sqrt{k^2-\eta^2})|}{|\xi - \eta |^{1+2b}} d\eta \right] \: d\xi
        = \int_{|\xi|\leq \sqrt a}+ \int_{\sqrt a \leq |\xi|\leq k} \Big[...\Big] d\xi  := A+B
    \]
    Let us start bounding $A$. Firstly, we claim that if $|\xi|\leq \sqrt a$ and $|\xi-\eta|\leq \sqrt a $, then it follows $|\sqrt{k^2-\xi^2}-\sqrt{k^2-\eta^2}|\leq a$. Indeed, we have
    \begin{align*}
        |\sqrt{k^2-\xi^2}-\sqrt{k^2-\eta^2}| & = \frac{|\eta^2-\xi^2|}{\sqrt{k^2-\xi^2}+\sqrt{k^2-\eta^2}} 
         \leq \frac{ |\xi-\eta|\: |\xi+\eta|}{\sqrt{k^2-a}} 
         \leq \frac{3 a}{\sqrt{k^2-a}} \leq a
    \end{align*}
    provided that $k^2-a\geq 9$, which we may assume since $a<1$ and $k$ is large. Therefore, we get
    \begin{align*}
        A & \leq \|\psi\|_\infty \int_{|\xi|\leq \sqrt a} |F(\xi)|^2 \Big[ \int_{|\xi-\eta|\geq \sqrt a} \frac{d\eta}{|\xi-\eta|^{1+2b}} \Big]\: d\xi 
        = \frac{\|\psi\|_\infty}{2b\: a^{b}} \int_{|\xi|\leq \sqrt a} |F(\xi)|^2 d\xi.
    \end{align*}
    Now we bound $B$. Whenever $|\xi|\geq \sqrt a $ and $|\xi-\eta|\leq \ {a^2}/({16\xi})$, then  $     |\sqrt{k^2-\xi^2}-\sqrt{k^2-\eta^2}|< a$. In order to prove this claim, we consider two cases. First, suppose that $|\xi|$ is close to $k$, more precisely, suppose $k^2-\xi^2 \leq (a/4)^2$. Then, if we call $s = |\xi-\eta|,$ we have
    \[|k^2-\eta^2 |\leq |k^2-\xi^2|+s^2 + 2s \xi \leq \frac{a^2}{16}+\frac{a^4}{256\xi^2}+\frac{2a^2}{16}\leq \Big(\frac{7a}{16}\Big)^2 , \]
    since $a<1$ and $|\xi|\geq \sqrt a.$ Therefore
    \begin{align*}
        |\sqrt{k^2-\xi^2}-\sqrt{k^2-\eta^2}| & \leq \sqrt{k^2-\xi^2}+\sqrt{k^2-\eta^2} 
         \leq \frac{a}{4}+ \frac{7a}{16} 
         <a.
    \end{align*}
    For the second case, if $|\xi|$ is far away from $k$, that is, $k^2-\xi^2 \geq (a/4)^2,$ then
    \begin{align*}
        |\sqrt{k^2-\xi^2}-\sqrt{k^2-\eta^2}| & = \frac{|\xi-\eta|\: |\xi+\eta|}{\sqrt{k^2-\xi^2}+\sqrt{k^2-\eta^2}} 
         \leq \frac{\frac{a^2}{16\xi} (\xi +\frac{a^2}{16\xi})}{a/4}= \frac{a}{4} \Big(1+\frac{a^2}{16\xi^2} \Big) \leq \frac{17a}{64}<a,
    \end{align*}
    where once again we used $|\xi|\geq \sqrt a$ and $a<1$. This proves the claim, and we can infer
    \begin{align*}
        B &\leq  \|\psi\|_\infty \int_{|\xi|\geq \sqrt a} |F(\xi)|^2 \Big[ \int_{|\xi-\eta|\geq  a^2/(16\xi)} \frac{d\eta}{|\xi-\eta|^{1+2b}} \Big]\: d\xi 
        \lesssim_b \frac{\|\psi\|_\infty}{ a^{4b}} \int_{|\xi|\leq k} |\xi|^{2b} |F(\xi)|^2 d\xi.
    \end{align*}
    Collecting all estimates, we have proven that 
    \[|\langle \hat h_b , \psi\rangle| \lesssim_b  \frac{\|\psi\|_\infty}{a^{4b}} \int_{|\xi|\leq k } (1+|\xi|^{2b}) |F(\xi)|^2d\xi.\]
    Note that the above integral is finite since $F\in L^2(-k,k)$ and by \eqref{eq:regularity-conditions-fractional}. 
\end{proof}
\end{demostraciones-dentro}

Moreover, using \eqref{eq:FT-of-hb-expression}, we can show that $\hat h_b$ is equal to a locally integrable function away from zero and provide an explicit formula for it.
\begin{proposition}\label{prop:FT-of-h-function}
	Let $h_b$ be defined as in \eqref{eq:fractional-dispersion-def}, with $F\in L^2(-k,k)$ satisfying \eqref{eq:regularity-conditions-fractional}. Then, away from the origin, $\hat h_b$ is equal to the function 
	\begin{equation}
		\label{eq:Phi-expresion}
		\Phi(\tau) = -\frac{\omega_b}{8} \sum_{m,n\in\{0,1\}} \int_{|\tau|}^{2k-|\tau|}    \Phi_{n,m}(\tau,v)  \frac{\tau^2-v^2}{\big(k^2-(\frac{\tau+v}{2})^2\big)^{\frac12}\big(k^2-(\frac{\tau-v}{2})^2\big)^\frac 12}\: dv,
			%
	\end{equation}
	where $\omega_b$ is given by \eqref{eq:omega-b-def} and, for $0\leq n,m\leq 1$, the kernel functions $\Phi_{n,m}$ are defined as
	\begin{align*}
		\Phi_{n,m}(\tau,v) = \frac{ F\Big((-1)^n\big(k^2-(\frac{\tau+v}{2})^2\big)^\frac12 \Big) \overline{ F \Big((-1)^m\big(k^2-(\frac{\tau-v}{2})^2\big)^\frac12\Big)}   }{\Big|(-1)^n\sqrt{k^2-(\frac{\tau+v}{2})^2}-(-1)^m\sqrt{k^2-(\frac{\tau-v}{2})^2} \Big|^{1+2b}}.
	\end{align*}
	By this, we mean that for a function $\psi\in \mathscr S(\R)$ vanishing on a neighborhood of zero, 
		\begin{align*}
			\langle \hat h_b ,\psi\rangle & = \int_{-k}^k \psi(\tau)\Phi(\tau)d\tau.
		\end{align*}
\end{proposition}
\begin{proof}
	In order to identify $\Phi$, we split the integration domain in \eqref{eq:FT-of-hb-expression} into the four quadrants:
	\[
	A_{0,0}=\{\xi,\eta\geq 0\},
    \quad
    A_{0,1}=\{\xi\geq 0,\eta\leq 0\},
    \quad
    A_{1,0}=\{\xi\leq 0,\eta \geq 0\},
    \quad
    A_{1,1}=\{\xi,\eta\leq 0\}.
	\]
	The proof follows easily by exploiting, in each of these quadrants, the change of variables
	\begin{equation*}
		\tau=\sqrt{k^2-\xi^2}-\sqrt{k^2-\eta^2},\qquad 
		v=\sqrt{k^2-\xi^2}+\sqrt{k^2-\eta^2}. \qedhere
	\end{equation*}
	\begin{demostraciones-fuera}
	In order to change the variables, we are missing just the jacobian:
	\begin{align*}
		J
		&=\left|\begin{matrix}\xi_\tau & \xi_v\\ \eta_\tau & \eta_v \end{matrix}\right| 
		= \frac 1{16} \left|\begin{matrix} \frac{\tau+v}{\sqrt{k^2-(\frac{\tau+v}{2})^2}} 
			& \frac{\tau+v}{\sqrt{k^2-(\frac{\tau+v}{2})^2}}\\
			 \frac{\tau-v}{\sqrt{k^2-(\frac{\tau-v}{2})^2}} &
			 \frac{v-\tau}{\sqrt{k^2-(\frac{\tau-v}{2})^2}} \end{matrix}\right|
		= \frac{v^2-\tau^2}{16\sqrt{k^2-(\frac{\tau+v}{2})^2}\sqrt{k^2-(\frac{\tau-v}{2})^2}}
	\end{align*}
	The transformation transforms each of the cuadrants $A_{n,m}$ into the square
	\[
	\{-k\leq \tau\leq k,|\tau|\leq v \leq 2k-|\tau|\}
	\]
	So, making the change of variables in \eqref{eq:FT-of-hb-expression}, we have
	\begin{align*}
		\langle \hat h_b,\psi\rangle &
		= -\frac{\omega_b}{8} \sum_{n,m=0}^1\int_{-k}^k\psi(\tau)\int_{|\tau|}^{2k-|\tau|} F\Big( (-1)^n\sqrt{k^2-\big(\frac{\tau+v}{2}\big)^2}\Big) \overline{F}\Big( (-1)^m\sqrt{k^2-\big(\frac{\tau-v}{2}\big)^2}\Big) \\
		& \qquad \qquad \frac{\tau^2-v^2}{\sqrt{k^2-(\frac{\tau+v}{2})^2}\sqrt{k^2-(\frac{\tau-v}{2})^2}}\: \frac{dv}{\Big|(-1)^n\sqrt{k^2-(\frac{\tau+v}{2})^2}-(-1)^m\sqrt{k^2-(\frac{\tau-v}{2})^2} \Big|^{1+2b}}\:d\tau
	\end{align*}
	\end{demostraciones-fuera}
\end{proof}

    \begin{proposition} \label{prop:singularity-h-hat} Let $\Phi$ be be given by \eqref{eq:Phi-expresion}, with $F\in L^2(-k,k)$ satisfying \eqref{eq:regularity-conditions-fractional}. Then
     	\[
		\lim_{\tau\rightarrow 0} |\tau|^{1+2b}\:|\Phi(\tau)| = \omega_b \int_{-k}^k  \frac{|\xi|^{2b}|F(\xi)|^2}{(k^2-\xi^2)^{b}} \: d\xi.		\]
    \end{proposition}

    \begin{demostraciones-dentro}
    \begin{proof}
    First, note that if $n\neq m$, the kernel has no singularity at $\tau=0$; hence, the contribution from this part is bounded, and its limit is zero. Therefore, we only need to consider the diagonal cases where $n=m$. An elementary computation shows 
    \begin{align*}
    	\lim_{\tau \rightarrow 0}\frac{|\tau|^{1+2b}}{\Big|\sqrt{k^2-(\frac{\tau+v}{2})^2}-\sqrt{k^2-(\frac{\tau-v}{2})^2} \Big|^{1+2b}} 
	%
	%
	%
	& = \Big(\frac{2}{|v|}\Big)^{1+2b} \Big( \sqrt{k^2-(v/2)^2}\Big)^{1+2b}.
    \end{align*}
    	We therefore have
	\begin{align*}
		\lim_{\tau\rightarrow 0}|\tau|^{1+2b}|\Phi(\tau)| 
		%
		 %
		 &={\omega_b}2^{2b-2}\int_{0}^{2k}\Big[ \sum_{n=0}^1\Big|F\Big((-1)^n\sqrt{k^2-\Big(\frac{v}{2}\Big)^2} \Big)\Big|^2\Big]
		 \:\frac{v^{1-2b}}{\big(k^2-(v/2)^2\big)^{\frac 12-b}}\: dv
	\end{align*}
	Making the change of variables $\xi = \sqrt{k^2-(v/2)^2}$, we get
	\begin{align*}
	\lim_{\tau\rightarrow 0}|\tau|^{1+2b} \:|\Phi(\tau)| 
	& =\omega_b\int_{0}^{k}\big[ |F(\xi)|^2+|F(-\xi)|^2 \Big]\: \frac{|\xi|^{2b}}{(\sqrt{k^2-\xi^2})^{2b}} \: d\xi,	
	\end{align*}
	which concludes the proof.
    \end{proof}
    \end{demostraciones-dentro}

    \begin{demostraciones-fuera}
    \begin{proof}
    First note that if $n\neq m$, the kernel has no singularity at $\tau=0$ and therefore the contribution from that part is bounded and its limit is zero. So, we only consider the diagonal cases $n=m$. We have
    \begin{align*}
    	\lim_{\tau \rightarrow 0}\frac{|\tau|^{1+2b}}{\Big|\sqrt{k^2-(\frac{\tau+v}{2})^2}-\sqrt{k^2-(\frac{\tau-v}{2})^2} \Big|^{1+2b}} 
	&= \lim_{\tau \rightarrow 0}|\tau|^{1+2b}\frac{\Big|\sqrt{k^2-(\frac{\tau+v}{2})^2}+\sqrt{k^2-(\frac{\tau-v}{2})^2} \Big|^{1+2b}}{\Big|\big(k^2-(\frac{\tau+v}{2})^2\big)-\big(k^2-(\frac{\tau-v}{2})^2\big) \Big|^{1+2b}} \\
	& = \Big|\sqrt{k^2-(\frac{v}{2})^2}+\sqrt{k^2-(\frac{v}{2})^2} \Big|^{1+2b} \lim_{\tau \rightarrow 0} \frac{|\tau|^{1+2b}}{|\frac{\tau v}{2}|^{1+2b}} \\
	& = \frac{1}{|v|^{1+2b}} \Big( \sqrt{k^2-(v/2)^2}\Big)^{1+2b}.
    \end{align*}
    	We therefore have
	\begin{align*}
		\lim_{\tau\rightarrow 0}|\tau|^{1+2b}|\Phi(\tau)| & = \frac{\omega_b}{8}\int_{0}^{2k}\Big[ \sum_{n=0}^1 F\Big((-1)^n\sqrt{k^2-(\frac{v}{2})^2} \Big) \overline{ F}\Big((-1)^n\sqrt{k^2-(\frac{v}{2})^2}\Big) \Big]
		 \:\frac{v^2\: \Big( \sqrt{k^2-(v/2)^2}\Big)^{1+2b}}{|v|^{1+2b}\big(k^2-(v/2)^2\big)}\: dv\\
		 &=\frac{\omega_b}{8}\int_{0}^{2k}\Big[ \sum_{n=0}^1\Big|F\Big((-1)^n\sqrt{k^2-(\frac{v}{2})^2} \Big)\Big|^2\Big]
		 \:\frac{v^{1-2b}}{\big(k^2-(v/2)^2\big)^{\frac 12-b}}\: dv
	\end{align*}
	
	Making the change of variables $\xi = \sqrt{k^2-(v/2)^2}$, we get
	\begin{align*}
	\lim_{\tau\rightarrow 0}|\tau|^{1+2b} \:|\Phi(\tau)| 
	& =2^{1-2b}\omega_b\int_{0}^{k}\big[ |F(\xi)|^2+|F(-\xi)|^2 \Big]\: \frac{\xi^{\frac{1+2b}{2}}}{(\sqrt{k^2-\xi^2})^{2b}} \: d\xi	\\
		& = {\color{blue}\frac{\omega_b}{4^b}}\int_{-k}^k |F(\xi)|^2\: \frac{|\xi|^{ {\color{blue}\frac{1+2b}{2}}}}{(\sqrt{k^2-\xi^2})^{{\color{blue}2b}}} \: d\xi.
	\end{align*}
    \end{proof}
    {\color{blue} Tenemos que comprobar que los exponentes  en azul y los coeficientes sean los correctos, que con los coeficientes no he sido muy cuidadoso, y los exponentes me chirrían un poco.}
    \end{demostraciones-fuera}

    \begin{proof}[Proof of Theorem \ref{thm:dynamical-fractional-UP}]
    By Propositions \ref{prop:FT-of-h-function} and \ref{prop:singularity-h-hat}, we have identified the singularity of $\hat h_b(\tau)$ in the origin. Therefore, the behavior of $h_b(y)$ at infinity concludes the proof of Theorem \ref{thm:dynamical-fractional-UP}.
\end{proof}


\section{Periodic scattering data for \texorpdfstring{$b<1$}{b<1}}
\label{seq:fractional-periodic}

We now consider periodic scattering data, that is $u_0(x+2\pi)=u_0(x)$. Therefore, $F=\hat u_0$ can be written as \eqref{eq:Truncated--Dirac-Comb-def}. We approximate $F$ with Schwartz functions in a similar way to the case $b=1$. Let $\varphi\in \mathscr S(\R)$ be non-negative, supported on the unit interval, with integral one, and  let $\varepsilon>0$ be a small parameter. We define
\begin{equation}
    \label{eq:f-epsilon}
    F^\varepsilon (\xi) = \frac{1}{\sqrt\varepsilon} \sum_{|n|<k} F(n) \varphi\Big(\frac{\xi-n}{\varepsilon}\Big).
\end{equation}
Let us compute $h_b[F^\varepsilon]$. By Theorem \ref{th:upper-bound-h-delta} we know that it is a tempered distribution, and by Theorem \ref{th:hb-transform-formula} we can compute, for $\psi \in \mathscr S(\R)$ supported away from zero, 
\begin{align*}
    \langle \hat{h_b}[F^\varepsilon],\psi \rangle
    %
    %
    & = \frac{-2\omega_b}{\varepsilon}\sum_{|n|,|m|<k} F(n)\overline{F(m)} \int_{\substack{|\xi-n|\leq \varepsilon\\ |\eta-m|\leq \varepsilon} } \varphi({ \textstyle\frac{{\xi-n}}{\varepsilon}}) \varphi({\textstyle \frac{\eta-m}{\varepsilon}}) \frac{\psi\big( \sqrt{k^2-\xi^2}-\sqrt{k^2-\eta^2}\big)}{|\xi-\eta|^{1+2b}}  {d\xi d\eta} .
\end{align*}
Comparing this formula with the non-fractional dispersion obtained in \eqref{eq:h1-periodic-bahaviour}, we observe that the different frequencies interact with each other. Therefore, we distinguish two parts in the sum above: the diagonal part $n=m$, which identify the singularity at the origin, and the off-diagonal part $n\neq m$, which identify the interaction among the different frequencies. We will refer to these as the \textit{singular} and \textit{regular parts} of $h_b$, respectively.

\begin{theorem}\label{thm:hb-periodic-decomposition}
	Let $F^\varepsilon$ be defined by \eqref{eq:f-epsilon}. Then, the dispersion $h_b[F^\varepsilon]$ can be decomposed as 
	\begin{equation}\label{eq:h-decomposition}
	h_b[F^\varepsilon](y) = S^\varepsilon_b[F](y) + \varepsilon R^\varepsilon_b[F](y),
	\end{equation}
	where the singular  part $S_b^\varepsilon$ is given by
	\begin{equation}
		\label{eq:singular-part-definition}
		S^\varepsilon_b[F](y) = \frac{1}{\varepsilon^{2b}} \sum_{|n|<k}|F(n)|^2 \int_{\R} |x|^{2b} \Big| \int_{(-1,1)} \varphi(\eta) e^{2 \pi ix\eta+2\pi i y\sqrt{k^2-(n+\varepsilon\eta)^2}}d\eta \Big|^2dx,
	\end{equation}
	and the regular part $R_b^\varepsilon$ is given by the action of its Fourier transform on $\psi \in \mathscr S(\R)$ as
	\begin{equation}
        \label{eq:eps-periodic-regular-part}
        \langle \hat R^\varepsilon_b[F],\psi \rangle =\frac{-2\omega_b}{\varepsilon^2}\sum_{\substack{|n|,|m|<k \\ n\neq m}} F(n)\overline{F(m)} \int_{\substack{|\xi-n|\leq \varepsilon\\ |\eta-m|\leq \varepsilon} } \varphi({ \textstyle\frac{{\xi-n}}{\varepsilon}}) \varphi({\textstyle \frac{\eta-m}{\varepsilon}}) \frac{\psi\big( \sqrt{k^2-\xi^2}-\sqrt{k^2-\eta^2}\big)}{|\xi-\eta|^{1+2b}}  {d\xi d\eta} .
    \end{equation}
\end{theorem}

    \begin{proof}Using \eqref{eq:representation-hb-Fractional}, we explicitly compute
        \begin{align*}
            h_b[F^\varepsilon](y) 
            %
            %
            %
            %
            & = \frac{\omega_b}{\varepsilon} \sum_{|n|,|m|<k} F(n)\overline {F(m)} \int_\R\int_\R \left( \varphi( {\textstyle \frac{\xi-n}{\varepsilon}})  e^{2 \pi i y \sqrt{k^2-\xi^2}}-\varphi( {\textstyle \frac{\eta-n}{\varepsilon}}) e^{2 \pi i y \sqrt{k^2-\eta^2}}\right) \\
            &  \qquad \qquad \qquad \qquad  \left( \varphi( {\textstyle \frac{\xi-m}{\varepsilon}})  e^{-2 \pi i y \sqrt{k^2-\xi^2}}-\varphi( {\textstyle \frac{\eta-m}{\varepsilon}}) e^{-2 \pi i y \sqrt{k^2-\eta^2}}\right)\frac{d\xi d\eta}{|\xi-\eta|^{1+2b}}.
        \end{align*}
        Denote by $S_b^\varepsilon(y)$ the sum over $n=m$ and by $\varepsilon R_b^\varepsilon(y)$ the off-diagonal sum over $n\neq m$. In this way, the decomposition \eqref{eq:h-decomposition} holds. Let us now identify both terms. We begin with the regular part. By disjointness of the supports, we have
        \begin{equation}\label{eq:R-epsilon-y}
            R_b^\varepsilon(y) = 
            -\frac{2\omega _b}{\varepsilon^2} \sum_{\substack{|n|,|m|<k \\ n \neq m}} F(n)\overline {F(m)} \int_{\substack{|\xi-n|\leq \varepsilon\\ |\eta-m|\leq \varepsilon}}  \varphi({\textstyle \frac{\xi-n}{\varepsilon}})\varphi ({\textstyle  \frac{\eta-m}{\varepsilon}})  e^{2\pi i y (\sqrt{k^2-\xi^2}-\sqrt{k^2-\eta^2})}\frac{d\xi d\eta}{|\xi-\eta|^{1+2b}}.
        \end{equation}
        Let us test the Fourier transform of $R_b^\varepsilon$ with a test function $\psi\in \mathscr S(\R)$. 
        We have
        \begin{align*}
            \langle \hat{R_b^\varepsilon}, \psi \rangle 
            %
            %
            & = -\frac{2\omega_b}{\varepsilon^2} \sum_{\substack{|n|,|m|<k \\ n \neq m}} F(n)\overline {F(m)} \int_{\substack{|\xi-n|\leq \varepsilon\\ |\eta-m|\leq \varepsilon}}  \frac{\varphi({\textstyle \frac{\xi-n}{\varepsilon}})\varphi ({\textstyle  \frac{\eta-m}{\varepsilon}})}{|\xi-\eta|^{1+2b}}  \psi \big( \sqrt{k^2-n^2}-\sqrt{k^2-m^2}\big) d\xi d\eta,
            %
        \end{align*}
        concluding the proof of \eqref{eq:eps-periodic-regular-part}. Now, we identify the singular part. In order to do this, let $\varphi_{n,\varepsilon}(\xi)=\varphi(\frac{\xi-n}{\varepsilon})e^{2\pi i y\sqrt{k^2-\xi^2}}$. Using \eqref{eq:representation-hb-Fractional} and making a change of variables, we get
        \begin{align*}
        	S_b^\varepsilon(y) %
			&= \frac{\omega_b}{\varepsilon} \sum_{|n|<k} |F(n)|^2 \int_{\R^2} \Big| \varphi\big( \frac{\xi-n}{\varepsilon}\big) e^{2\pi i y\sqrt{k^2-\xi^2}}-\varphi\big( \frac{\eta-n}{\varepsilon}\big) e^{2\pi i y\sqrt{k^2-\eta^2}} \Big|^2 \frac{d\xi d\eta }{|\xi-\eta|^{1+2b}} \\
			& = \frac 1\varepsilon  \sum_{|n|<k} |F(n)|^2 \int_\R |x|^{2b} | \mathcal F^{-1}\varphi_{n,\varepsilon}(x)|^2dx \\
			& = \frac1\varepsilon  \sum_{|n|<k} |F(n)|^2 \int_\R |x|^{2b} \Big|\int_{\R} \varphi \big( \frac{\xi-n}{\varepsilon}\big) e^{2\pi i y\sqrt{k^2-\xi^2}} e^{2\pi i x \xi}d\xi\Big|^2dx\\
			& = \varepsilon  \sum_{|n|<k} |F(n)|^2 \int_\R |x|^{2b} \Big|\int_{-1}^1 \varphi (\eta) e^{2\pi i y\sqrt{k^2-(\eta\varepsilon + n)^2}} e^{2\pi i x (\eta\varepsilon + n)}d\eta\Big|^2dx,
        \end{align*}
        which is precisely \eqref{eq:singular-part-definition} upon changing the variable $\varepsilon x$ to $x$.
    \end{proof}
    
    It can be easily checked from \eqref{eq:singular-part-definition} that the singular part grows as $\varepsilon^{-2b}$ when $\varepsilon \rightarrow 0$, which was expected in analogy with \eqref{eq:h1-periodic-bahaviour}. However, the regular part actually converges to the analytic function $R_b$ in \eqref{eq:periodic-limit-function}.
    Note that if $k,n,m$ are integers, then the frequencies of $R_b$ are not integers and therefore $R_b$ is not a periodic function.
    \begin{theorem}\label{thm:convergence-periodic-part}
    The regular part $R_b^\varepsilon$ given by \eqref{eq:eps-periodic-regular-part} tends to the function $R_b$ in \eqref{eq:periodic-limit-function} as $\varepsilon \rightarrow 0$, both in $\mathscr S'(\R)$ and uniformly in compact sets.
    \end{theorem}

\begin{lemma}
    \label{prop:properties-of-h}
    If $\varepsilon<\frac 14$, then $\hat R_b^\varepsilon[F]\in L^1(\R)$ with norm uniformly bounded by
    \begin{equation}
        \label{eq:norm-h-eps-b}
        \|\hat R_b^\varepsilon[F]\|_{L^1(\R)}\leq C_b \sum_{\substack{|n|,|m|<k \\ n\neq m}} \frac{|F(n)F(m)|}{|n-m|^{1+2b}}.
    \end{equation}
\end{lemma}

\begin{proof}
    We claim that, for any $\psi \in \mathscr S( \R)$, the following bound holds:
    \begin{equation}
        \label{eq:claim-h-periodic}
        |\langle \hat R_{b}^\varepsilon[F],\psi \rangle |\leq  C_b \Big( \sum_{\substack{|n|,|m|<k \\ n\neq m}} \frac{|F(n)F(m)|}{|n-m|^{1+2b}}\Big) \: \| \psi \|_\infty.
    \end{equation}
    Indeed, by making a change of variables, we have that, if $m\neq n$,
    \begin{align*}
        \frac 1{\varepsilon ^2} \int_{\substack{|\xi-n|\leq \varepsilon\\ |\eta-m|\leq \varepsilon}  }  & \Big|\varphi\Big(\frac{\xi-n}{\varepsilon}\Big) \varphi\Big(\frac{\eta-m}{\varepsilon}\Big) \psi\big( \sqrt{k^2-\xi^2}-\sqrt{k^2-\eta^2}\big) \Big| \frac{d\xi d\eta}{|\xi-\eta|^{1+2b}} \\
        & \leq \frac{\|\psi\|_\infty }{(|n-m|-2\varepsilon)^{1+2b}} \int_{\substack{|\xi-n|\leq \varepsilon\\ |\eta-m|\leq \varepsilon}  } \Big|\varphi\Big(\frac{\xi-n}{\varepsilon}\Big) \varphi\Big(\frac{\eta-m}{\varepsilon}\Big) \Big| \frac{d\xi d\eta}{\varepsilon^2} 
        %
         \leq  \frac{2^{1+2b} \|\psi\|_\infty }{|n-m|^{1+2b}}.
    \end{align*}
    Inserting this in \eqref{eq:eps-periodic-regular-part} we obtain \eqref{eq:claim-h-periodic}, thus proving the claim. Since $\mathscr S(\R)$ is dense in $C_0(\R)$, the space of continuous functions tending to zero at infinity, we find that \eqref{eq:claim-h-periodic} also holds for all $\psi\in C_0(\R)$ and so, by the Riesz--Markov theorem, $\mu_\varepsilon=\hat R_b^\varepsilon[F]$ is a signed measure in $\R$ with total variation bounded by the right-hand side of \eqref{eq:norm-h-eps-b}. Note also that \eqref{eq:eps-periodic-regular-part} extends naturally to $\psi \in L^\infty(\R)$. Therefore, if $A\subset \R$ is a Borel set with zero measure, 
    \[\mu_\varepsilon(A) = \langle \hat R_b^\varepsilon[F],\chi_A \rangle =0, \]
    which implies that $\mu_\varepsilon$ is absolutely continuous with respect to Lebesgue measure and thus, by the Radon--Nikodym theorem, an integrable function that satisfies \eqref{eq:norm-h-eps-b}.
    \end{proof}

    \begin{proof}[Proof of Theorem \ref{thm:convergence-periodic-part}]
		First of all, from \eqref{eq:periodic-limit-function}, the Fourier transform of $R_b$ is 
		\begin{equation}
			\label{eq:fourier-transform-of-the-limit}
			\hat R_b(\tau)= -2\omega_b \sum_{\substack{|n|,|m|<k \\ n\neq m}} \frac{F(n)\overline{ F(m)}}{|n-m|^{1+2b}}\delta \Big( \tau- \big(\sqrt{k^2-n^2}-\sqrt{k^2-m^2}\big)\Big).
		\end{equation}  
		Let us expand the Fourier transform of $R_b^\varepsilon$ in \eqref{eq:eps-periodic-regular-part} as follows. Since $\varphi$ has integral one, for $\psi \in \mathscr S (\R)$, we have
        \begin{align}\label{eq:periodic-limit-estimate-1}
            \langle \hat R_b^\varepsilon ,\psi\rangle = 
            \langle \hat R_b,\psi\rangle  - 2c_b \sum_{\substack{|n|,|m|<k \\ n\neq m}}&F(n)\overline{F(m)} \int_{\substack{|\xi-n|\leq \varepsilon \\ |\eta-m|\leq \varepsilon}} \frac{1}\varepsilon \varphi  \Big( \frac{\xi-n}{\varepsilon}\Big) \frac{1}\varepsilon \varphi \Big( \frac{\eta-m}{\varepsilon}\Big)  \\
            \notag &\left[
                \frac{\psi(\sqrt{k^2-\xi^2}-\sqrt{k^2-\eta^2})}{|\xi-\eta|^{1+2b}}- \frac{\psi(\sqrt{k^2-n^2}-\sqrt{k^2-m^2})}{|n-m|^{1+2b}}
            \right] 
            d\xi d\eta.
        \end{align}
        Using elementary calculus and the fact that $|\xi-n|,|\eta-m|\leq \varepsilon$, we can bound the term in the square brackets in \eqref{eq:periodic-limit-estimate-1} as follows:
        \begin{equation}
            \label{eq:periodic-limit-claim-1} 
            \left|
                \frac{\psi(\sqrt{k^2-\xi^2}-\sqrt{k^2-\eta^2})}{|\xi-\eta|^{1+2b}}- \frac{\psi(\sqrt{k^2-n^2}-\sqrt{k^2-m^2})}{|n-m|^{1+2b}}
            \right|  \leq \varepsilon \: C_{b,k} \frac{\|\psi\|_\infty + \|\psi'\|_\infty}{|n-m|^{1+2b}}.
        \end{equation}
        Inserting \eqref{eq:periodic-limit-claim-1} into \eqref{eq:periodic-limit-estimate-1}, we have
        \begin{align*}
            |\langle \hat R_b^\varepsilon ,\psi\rangle-\langle \hat R_b,\psi\rangle| & 
            \leq \varepsilon\: C_{b,k}(\|\psi\|_\infty+\|  \psi'\|_\infty ) \: \sum_{\substack{|n|,|m|<k \\ n\neq m}}\frac{|F(n)F(m)|}{|n-m|^{1+2b}} \int_{\substack{|\xi-n|\leq \varepsilon \\ |\eta-m|\leq \varepsilon}}  \varphi  ({\textstyle \frac{\xi-n}{\varepsilon}})  \varphi ({\textstyle  \frac{\eta-m}{\varepsilon}}) \frac{d\eta d\xi}{\varepsilon^2}  \\
            &  = \varepsilon\: C_{b,k}(\|\psi\|_\infty  +\|\psi'\|_\infty ) \: \sum_{\substack{|n|,|m|<k \\ n\neq m}}\frac{|F(n)F(m)|}{|n-m|^{1+2b}}.
        \end{align*}
        This proves that $\hat R_b^\varepsilon$ converges to $\hat R_b$ in $\mathscr S'(\R)$ as $\varepsilon$ tends to zero. Since the Fourier transform is an  isometry in $\mathscr S'(\R)$, we also have that $ R_b^\varepsilon$ tends to $R_b$ in $\mathscr S' (\R)$. In order to show that the limit is uniform in compact sets, one can check that the limit is uniform over compact sets in each of the terms of \eqref{eq:R-epsilon-y} by Taylor expansion.       
        \end{proof}

\subsection{Limit as \texorpdfstring{$k$}{k} tends to infinity}

When the scattering data is periodic, we have identified the regular part of the dispersion as \eqref{eq:periodic-limit-function}. We show that, as the scaling factor $k$ tends to infinity, this function converges to the distribution $\mathsf h_{b,\text{per}}$ in \eqref{eq:hd-Schrodinger}, which was obtained in \cite{MR4469234} using the Dirac comb in the Schr\"odinger equation. The Fourier transfor of $\mathsf h_{b,\text{per}}$ is
\begin{equation}
	\label{eq:hd-Schrodinger-FT}
	\hat {\mathsf h}_{b,\text{per}}(\tau) = -2\omega_b \sum_{r\in \Z }\delta(\tau-\frac r2 ) \Big( \sum_{\substack{n\neq m\\n^2-m^2=r}}\frac{1}{|n-m|^{1+2b}}\Big)=-2\omega_b \sum_{r\in \Z} \alpha_b(r) \delta(\tau-\frac r2),
\end{equation}
with coefficients $\displaystyle \alpha_b(r) = 2\sum_{ {d|r,\ d>0}} d^{-1-2b}$ if $r$ is odd, $\displaystyle \alpha_b(r) = 2^{-2b}\sum_{{ 4d|r,\ d>0}} d^{-1-2b}$ if $r$ is multiple of 4, and $\alpha_b(r)=0$ otherwise.

Fix the scale $k>0$ and denote by $R_b^{(k)}$ the function \eqref{eq:periodic-limit-function}, that is, the regular limit of the dispersion with scattering data $F(\xi)=\sum_{|n|<k}F(n)\delta(\xi-n)$. We will specify the coefficients $F(n)$ later. By comparing the expression of $\hat R_b^{(k)}$ in \eqref{eq:fourier-transform-of-the-limit} with the expression of $\hat {\mathsf h}_{b,\text{per}}$ in \eqref{eq:hd-Schrodinger-FT}, we observe that we need to explore the Taylor expansion of the square root. An elementary computation shows
\begin{demostraciones-fuera}
\begin{align*}
    \sqrt{k^2-n^2}-\sqrt{k^2-m^2} & = k \left( \sqrt{1-(n/k)^2}-\sqrt{1-(m/k)^2} \right) \\
    & = k \left(1- \frac{n^2}{2k^2} + O(\frac{n^4}{k^4}) -  1 +  \frac{m^2}{2k^2} + O(\frac{m^4}{k^4})\right) \\
    & = \frac{m^2-n^2}{2k} + O(\frac{n^4}{k^3}) +  O(\frac{m^4}{k^3})
\end{align*}
\end{demostraciones-fuera}
\begin{demostraciones-dentro}
\begin{align*}
    \sqrt{k^2-n^2}-\sqrt{k^2-m^2} = \frac{m^2-n^2}{2k} + O(\frac{n^4}{k^3}) +  O(\frac{m^4}{k^3}), \qquad k\rightarrow \infty.
\end{align*}
\end{demostraciones-dentro}
So, we can write $\hat R_b^{(k)}$ as 
\begin{align*}
    \hat R_b^{(k)}(\tau) 
    %
    %
    & = -2\omega_b \: k\sum_{\substack{|n|,|m|<k \\ n\neq m}} \frac{F(n)\overline{F(m)}}{|n-m|^{1+2b}} \: \delta\Big( k{\tau} - \frac{m^2-n^2}{2} + O(\frac{n^4}{k^2}) +  O(\frac{m^4}{k^2})\Big),
\end{align*}
where we used the homogeneity of the Dirac delta. Fix $\varepsilon>0$ small and set $F(n)=1$ if $|n|\leq k^{\frac12-\varepsilon}$, and $F(n)=0$ otherwise. This way, upon rescaling $\hat R_b^{(k)}$, we obtain the asymptotic behavior
\begin{equation}
\label{eq:periodic-part-limit-target}
    \hat{ \mathcal{R}}_b^{(k)}(\tau) = \frac 1k \hat R_b^{(k)}(\tau/k) = -2\omega_b\sum_{\substack{|n|,|m|<k^{\frac 12-\varepsilon} \\ n\neq m}} \frac{1}{|n-m|^{1+2b}} \delta\Big( \tau - \frac{m^2-n^2}{2} + O(k^{-4\varepsilon})\Big), \quad k\rightarrow \infty.
\end{equation}

\begin{comentario}
We need to prove that $H_k$ in \eqref{eq:periodic-part-limit-target} tends to $H$ in \eqref{eq:periodic-part-schrodinger} as $k\rightarrow \infty$ in the correct sense.

First, we check  that the convergence actually happens in the space of tempered distributions. Let $\psi \in \mathscr S$ be a test function. First, note that the following series is absolutely convergent:
    \begin{equation}
        \label{eq:series-of-coefficients}
        \sum_{n\neq m} \frac{|\psi(m^2-m^2)|}{|n-m|^{1+2b}} = \sum_{l,r; r\neq 0} \frac{|\psi(lr)|}{|r|^{1+2b}} \leq c_\psi \sum_{l,r; r\neq 0}\frac{1}{(1+l^2)|r|^{1+2b}} <+\infty.
    \end{equation}
Therefore, we can test $H_k-H$ against $\psi$ in the Fourier side to get
\begin{align*}
    \langle \hat H-\hat H_k,\psi\rangle %
    & = \sum_{\substack{|n|,|m|\leq k^{\frac 12 - \varepsilon}\\ n\neq m}} \frac{\psi\big(n^2-m^2\big)-\psi\big(n^2-m^2+O(k^{-4 \varepsilon})\big)}{|n-m|^{1+2b}}  + \sum_{\substack{|n|,|m|\geq k^{\frac 12 - \varepsilon}\\ n\neq m}} \frac{\psi\big(n^2-m^2\big)}{|n-m|^{1+2b}}
\end{align*}
One can easily check that both terms tend to zero by \eqref{eq:series-of-coefficients} and the mean value theorem. Therefore, we have actually checked that the correct limit of $H_k$ is $H$. Let us improve this limit to Sobolev spaces.
\end{comentario}
\begin{theorem}\label{thm:limit-k}
    \label{thm:conv-limit-k}
    The function
    $\mathcal R_b^{(k)}$ converges to the tempered distribution $\mathsf h_{b, \text{per}}$ in \eqref{eq:hd-Schrodinger} as $k\rightarrow\infty$ in the weighted Sobolev space $(1+t^2)^{-1}H^s(\R)$ for $s<-\frac 12$.
\end{theorem}
Note that we have to use the weighted Sobolev space $(1+t^2)^{-1}H^s(\R)$ because $\mathsf h_{b, \text{per}}$ is a periodic distribution, while the $\mathcal R_b^{(k)}$ are not.

\begin{proof}[Proof of Theorem \ref{thm:conv-limit-k}]
    First, note that the Fourier transform of $(1+t^2)^{-1}$ is a constant multiple of $e^{-|\tau|}$. Therefore, we need to check that
    \begin{equation}
        \label{eq:limit-k-need-to-check}
        \lim_{k\rightarrow\infty}\Big\| \big( (\hat{\mathcal R}_b^{(k)} - \hat{ \mathsf h}_{b,\text{per}})*e^{-|\tau|} \big)\: (1+\tau^2)^{s/2}\Big\|_{L^2(\R)} =0.
    \end{equation}
    We compute
    \begin{align*}
        (\hat{ \mathcal R}_b^{(k)} -\hat{ \mathsf h}_{b,\text{per}})*e^{-|\tau|} &
        = -2\omega_b\sum_{\substack{|n|,|m|\leq k^{\frac 12 - \varepsilon} \\ n\neq m}} \frac{1}{|n-m|^{1+2b}} \Big( e^{-|\tau-(n^2-m^2)|}-e^{-|\tau-(n^2-m^2)+O(k^{-4\varepsilon})|}\Big) \\
        &\qquad -2\omega_b \sum_{\substack{|n|,|m|\geq k^{\frac 12 - \varepsilon} \\ n\neq m}} \frac{1}{|n-m|^{1+2b}}e^{-|\tau-(n^2-m^2)|}.
    \end{align*}
    Since the functions $e^{-|\tau-r|}$ are Lipschitz with uniform constant, by dominated convergence we obtain \eqref{eq:limit-k-need-to-check} because, for $s<-\frac 12$,
    \[
    \int_{\R} \left| \sum_{n\neq m} \frac{1}{|n-m|^{1+2b}}e^{-|\tau-(n^2-m^2)|} \right|^2(1+\tau^2)^s d\tau <\infty,
    \]
    since the sum above is actually an $L^\infty$ function of $\tau$. This completes the proof.
    \begin{demostraciones-fuera}Indeed, we can write
    \begin{align*}
         \left|\sum_{n\neq m} \frac{1}{|n-m|^{1+2b}}e^{-|\tau-(n^2-m^2)|}\right| & =  \left|\sum_{r\in \Z} K_b(r) e^{-|\tau-r|}\right| \leq c_b\sum_{r\in \Z}e^{-|\tau-r|},
    \end{align*}
    since the coefficients $K_b(r)$ are uniformly bounded. The series in the right hand side is 1-periodic, bounded and uniformly convergent. This finishes the proof.
    \end{demostraciones-fuera}
\end{proof}
    

\section*{Acknowledgements}

J. Canto is supported  by the projects IT1615-22 of the Basque Government and   PID2023-146646NB-I00 funded by MICIU/AEI/10.13039/501100011033.

N. M. Schiavone is supported by the project PID2021-123034NB-I00/MCIN/AEI/10.13039/ 501100011033 funded by the Agencia Estatal de Investigación. 
This work began when he was affiliated to BCAM in Bilbao, where he was supported: by the grant FJC2021-046835-I funded by the EU \lq\lq NextGenerationEU''/PRTR and by MCIN/AEI/10.13039/501100011033; by the Basque Government through the BERC 2022--2025 program; and by the Spanish Agencia Estatal de Investigación through BCAM Severo Ochoa excellence accreditation CEX2021-001142-S/MCIN/AEI/10.13039/501100011033.
He is also member of the \lq\lq Gruppo Nazionale per L'Analisi Matematica, la Probabilit\`{a} e le loro Applicazioni'' (GNAMPA) of the \lq\lq Istituto Nazionale di Alta Matematica'' (INdAM).

L. Vega is  partially supported by MICIN (Spain) CEX2021-001142, PID2021-126813NB-I00 (ERDF A way of making Europe) and IT1615-22 (Gobierno Vasco).


\bibliographystyle{abbrv} 
\bibliography{CantoSchiavoneVega_AFUP}

\end{document}